\documentclass[review,onefignum,onetabnum]{siamart171218}

\usepackage{lipsum}
\usepackage{amsfonts}
\usepackage{graphicx}
\usepackage{epstopdf}
\usepackage{algorithmic}
\ifpdf
  \DeclareGraphicsExtensions{.eps,.pdf,.png,.jpg}
\else
  \DeclareGraphicsExtensions{.eps}
\fi

\newsiamremark{remark}{Remark}
\newsiamremark{hypothesis}{Hypothesis}
\crefname{hypothesis}{Hypothesis}{Hypotheses}
\newsiamthm{claim}{Claim}

\title{A Single-Mode Quasi Riemannian Gradient Descent Algorithm for Low-Rank Tensor Recovery\thanks{Submitted to the editors DATE.
}}

\author{Yuanwei Zhang
\thanks{School of Mathematical Sciences, Shanghai Jiao Tong University, Shanghai 200240, CHINA
  (\email{sjtuzyw@sjtu.edu.cn}).}
\and YA-NAN ZHU
\and Xiaoqun Zhang\thanks{School of Mathematical Sciences and Institute of Natural Sciences, Shanghai Jiao Tong University, Shanghai 200240, CHINA
  (\email{xqzhang@sjtu.edu.cn}).}}

\usepackage{amsopn}

\makeatletter
\newcommand*{\addFileDependency}[1]{% argument=file name and extension
  \typeout{(#1)}% latexmk will find this if $recorder=0 (however, in that case, it will ignore #1 if it is a .aux or .pdf file etc and it exists! if it doesn't exist, it will appear in the list of dependents regardless)
  \@addtofilelist{#1}% if you want it to appear in \listfiles, not really necessary and latexmk doesn't use this
  \IfFileExists{#1}{}{\typeout{No file #1.}}% latexmk will find this message if #1 doesn't exist (yet)
}
\makeatother

\ifpdf
\hypersetup{
  pdftitle={A Single-Mode Quasi Riemannian Gradient Descent Algorithm for Low-Rank Tensor Recovery},
  pdfauthor={ }
}
\fi

\usepackage{amsmath}
\usepackage{amsfonts}
\usepackage{mathrsfs}
\usepackage{color}
\usepackage{subfigure}
\usepackage{graphicx}
\usepackage{algorithm}
\usepackage{algorithmic}

\newcommand{\argmin}{\operatorname{argmin}}
\newcommand{\rank}{\operatorname{rank}}

\newcommand{\mcal}{\mathcal}
\newcommand{\scr}{\mathscr}
\newcommand{\bb}{\mathbb}
\newcommand{\bd}{\boldsymbol}
\newcommand{\mbf}{\mathbf}
\newcommand{\bnm}{\big\|}

\begin{document}

\maketitle
% REQUIRED
\begin{abstract}
This paper focuses on recovering a low-rank tensor from its incomplete measurements. We propose a novel algorithm termed the Single Mode Quasi Riemannian Gradient Descent (SM-QRGD). By exploiting the benefits of both fixed-rank matrix tangent space projection in Riemannian gradient descent and sequentially truncated high-order singular value decomposition (ST-HOSVD), SM-QRGD achieves a much faster convergence speed than existing state-of-the-art algorithms. Theoretically, we establish the convergence of SM-QRGD through the Tensor Restricted Isometry Property (TRIP) and the geometry of the fixed-rank matrix manifold. Numerically, extensive experiments are conducted, affirming the accuracy and efficacy of the proposed algorithm.

% This paper addresses the low-rank tensor recovery problem, focusing on reconstructing tensors with low multilinear rank from substantially incomplete measurements. We propose a novel algorithm, single mode quasi Riemannian gradient descent (SM-QRGD), which enhances computational efficiency compared to traditional methods like iterative hard thresholding and Riemannian gradient descent, specifically in thresholding tensors into low multilinear rank sets. This improvement is achieved by combining fixed-rank matrix tangent space projection and sequentially truncated high-order singular value decomposition (ST-HOSVD). Additionally, we offer recovery guarantees for SMQRGD, whether employing constant or normalized step sizes, based on the tensor restricted isometry Property (TRIP) and the geometry of fixed-rank matrix manifolds. Our extensive numerical experiments validate the superior efficacy of our algorithm compared to current methods.

\end{abstract}

% REQUIRED
\begin{keywords}
  low-rank tensor recovery, Riemannian manifold, iterative hard thresholding
\end{keywords}

% REQUIRED
\begin{AMS}

\end{AMS}

\section{Introduction}
Tensor is an extension of matrix into multiple dimensions, offering significant utility in data analysis across diverse domains such as computer vision \cite{li2010tensor, liu2012tensor}, machine learning \cite{anandkumar2014tensor}, signal processing \cite{nion2010tensor, cichocki2015tensor}, bioinformatics \cite{troyanskaya2001missing}, and quantum state tomography \cite{gross2010quantum}. This paper focuses on the tensor recovery problem, which aims to reconstruct an unknown tensor from its highly incomplete measurements.
To be concrete, for a given linear measurement operator $\scr{A}$ with the measurement $\mbf{y}\in \bb{R}^{m}$ of an unknown tensor $\bd{\mcal{T}} \in \bb{R}^{n_1 \times n_2 \times \cdots \times n_d}$, we aim to reconstruct $\bd{\mcal{T}}$ by solving the following optimization problem 
\begin{equation}\label{model: tensor recovery}
    \min_{{\mbf{\mcal{X}}}\in \bb{R}^{n_1\times \dots \times n_d}}\ \frac{1}{2}\bnm\scr{A}(\bd{\mcal{X}}) - \mbf{y} \bnm_2^2.
\end{equation}

In various applications, the dimension $m$ is considerably smaller than $\prod_{i = 1}^d n_i$, making \eqref{model: tensor recovery} an ill-posed problem. Usually, there is no unique solution if no additional prior knowledge is imposed. Analogous to the matrix case, one commonly adopted prior knowledge is the low rankness of $\bd{\mcal{T}}$. 
However, unlike the matrix, the rank of a given tensor is not unique. It often varies depending on the specific applications, for example, CANDECOMP/PARAFAC (CP) decomposition \cite{harshman1970foundations, carroll1970analysis}, Tucker decomposition \cite{tucker1966some} and tensor train (TT) decomposition \cite{oseledets2011tensor}. In this study, we focus on the Tucker decomposition and assume that the target tensor $\bd{\mcal{T}}$ possesses a low-multilinear-rank property, e.g., $\operatorname{rank}_T(\bd{\mcal{T}}) = \mathbf{r} = (r_1,\dots, r_d)$:
\begin{equation}\label{equ: low-rank assupmtion}
\bd{\mcal{T}} = \bd{\mcal{C}}\times_1\mbf{U}_1\times\cdots\times_d \mbf{U}_d,
\end{equation}
where $\bd{\mcal{C}}\in \bb{R}^{r_1\times \dots\times r_d}$ is the core tensor, $\mbf{U}_k$ is an $n_k$ by $r_k$ matrix with orthonormal columns, $\times_k$ is the tensor-matrix product along mode-$k$ (Refer to section \ref{subsection: Tensor operation} for more details).

In the context of operator $\scr{A}$, different selections of $\scr{A}$ are specifically designed to address different applications. These applications cover a broad spectrum of fields, from 3D medical imaging \cite{zhou2013tensor} to video sequences \cite{miao2020low} and recommendation systems \cite{bi2018multilayer}.
For instance, one commonly encountered challenge is the tensor completion problem. In such scenarios, the operator $\scr{A}$ acts as the sampling operator \cite{fornasier2011low, rauhut2017low}. Mathematically, the objective is to reconstruct a low-multilinear-rank tensor $\bd{\mcal{T}}$ from partially observed entries given by $\bd{\mcal{T}}(i_1, \dots, i_d)$, where the indices $(i_1, \dots, i_d)$ belong to a set $\Omega$. The tensor recovery problem with low-multilinear-rank prior can be formulated as the following optimization problem:
\begin{equation}\label{model: low-rank tensor recovery}
    \begin{aligned}
    \min_{{\mbf{\mcal{X}}}\in \bb{R}^{n_1\times \dots \times n_d}}
    &\ \frac{1}{2}\bnm\scr{A}(\bd{\mcal{X}}) - \mbf{y} \bnm_2^2\\
    \text{s.t.}\ &\operatorname{rank}_T(\bd{\mcal{X}}) = \mbf{r}.
    \end{aligned}
\end{equation}

Various methods have been introduced to address the problem presented in \eqref{model: low-rank tensor recovery}. Convex approaches, such as nuclear norm minimization, have been proven effective in low-rank matrix recovery.  In the realm of tensor recovery, similar approaches have been explored, including the strategy of minimizing the sum of nuclear norms (SNN) for unfolded matrices, as exemplified in \cite{gandy2011tensor, huang2014provable,huang2015provable,yuan2016tensor}.
However, the methods based on unfolding do not fully leverage the inherent tensor structure, which results in suboptimal sample complexity. Furthermore, tensor nuclear norm minimization is computationally demanding, especially given the NP-hard nature of evaluating tensor norms \cite{hillar2013most}.
Non-convex approaches, on the other hand, have gained increasing prominence, primarily due to their enhanced computational efficiency and better performance in terms of sampling complexity. These methods, based on the factorized form of tensor, apply alternating minimization \cite{zhou2013tensor, jain2014provable, xia2021statistically} and gradient descent \cite{xia2019polynomial, tong2022scaling, han2022optimal} to the factor matrices and core tensor.
Another line of non-convex approaches directly solves \eqref{model: low-rank tensor recovery} by using iterative hard thresholding (IHT) (or projected gradient descent) \cite{rauhut2017low, de2017low, chen2019non, ahmed2020tensor} or Riemannian optimization algorithms \cite{kressner2014low, cai2020provable, luo2021low, wang2021implicit}.
For example, \cite{rauhut2017low} first proposed the tensor iterative hard thresholding (TIHT) algorithm and provided the recovery guarantees based on tensor restricted isometry property (TRIP). However, the thresholding operator, truncated high order singular value decomposition (T-HOSVD), involves $d$ times SVD computation on the unfolded matrices. To alleviate the computational overload, \cite{de2017low} proposed an algorithm called sequentially optimal modal projection iterative hard thresholding (SeMPIHT). The approach reduces the computational complexity by sequentially decreasing the dimension of each mode through the application of sequentially truncated high-order singular value decomposition (ST-HOSVD) \cite{vannieuwenhoven2012new}. 
The Riemannian Gradient Descent (RGD) addresses \eqref{model: low-rank tensor recovery} from the perspective of Riemannian optimization. In the framework of IHT, RGD introduces an additional tangent space projection before the thresholding operation. This extra step mitigates the computational expense of the ensuing thresholding process. Nevertheless, the tangent space projection for a low-multilinear-rank tensor manifold necessitates the computation of $d + 1$ orthogonal components. Notably, the computational complexity of this process is in the same order as the direct thresholding operations.

In this work, we propose a new algorithm, which we call SM-QRGD, that takes advantage of the sequential thresholding operation of SeMPIHT and the tangent space projection of RGD so that a faster convergence can be achieved for solving the tensor recovery problem \eqref{model: low-rank tensor recovery}. 
In contrast to the existing Riemannian optimization literature, which treats the low-multilinear-rank tensor set as a Riemannian manifold, our innovation lies in viewing the modal truncated SVD as a retraction onto a manifold of low-rank matrices. By harnessing matrix tangent space techniques, we efficiently diminish the computational cost of the first-mode matrix truncated SVD in the SeMPIHT algorithm. Theoretically, we establish the convergence of SM-QRGD under the classical assumptions, i.e., TRIP of the operator $\scr{A}$. Numerical experiments in comparison with the TIHT algorithm \cite{rauhut2017low}, the SeMPIHT algorithm \cite{de2017low}, and the RGD algorithm \cite{cai2020provable} also validate the correctness and effectiveness of the method.

\textbf{Organizations.}
This paper is organized as follows. In Section \ref{section: Preliminary}, we introduce the notations, tensor operations, and associated algorithms. In Section \ref{section: Algorithm and Main Results}, we present our proposed algorithm and its convergence results. The numerical results are provided in Section \ref{section: numerical experiments}. We give our conclusion and future direction in Section \ref{section: Discussion and Future Direction}. Some useful lemmas and proofs of convergence are collected in Section \ref{section: Proof}.

\section{Preliminary}\label{section: Preliminary}
\subsection{Notations and Tensor Operations}\label{subsection: Tensor operation}
% \noindent\textbf{Notations.}
This paper uses capital calligraphic letters to represent tensors, capital letters for matrices, and lowercase letters for vectors. For instance, we denote a real $d$-order tensor as $\bd{\mcal{X}}\in \bb{R}^{n_1\times \dots \times n_d}$, a real matrix of dimensions $m\times n$ as $\mbf{X}\in \bb{R}^{m\times n}$, and a real vector with a length of $n$ as $\mbf{x}\in \bb{R}^n$. The entry in the $(j_1, \dots, j_d)$-th position of tensor $\bd{\mcal{X}}$ is represented as $\bd{\mcal{X}}_{j_1\dots j_d}$.

In the sequel, following the terminology of \cite{kolda2009tensor}, we briefly introduce the following basic tensor operations: 

\begin{itemize}
    \item \textit{Tensor matricization.}
    Tensor matricization is the process of converting a tensor into a matrix form. For a $d$-order tensor $\bd{\mcal{X}}\in \bb{R}^{n_1\times \dots\times n_d}$, we denote the mode-$k$ matricization operator as $\scr{M}_k$, and $\scr{M}_k(\bd{\mcal{X}})$ is a matrix of size $n_k\times \prod_{j = 1, j\neq k}^d n_j$.
    The element at position $(i_1, \dots, i_d)$ in the tensor corresponds to the entry at position $(i_k, j)$ in the matrix $\scr{M}_k(\bd{\mcal{X}})$, with
    \begin{equation*}
    j = 1 + \sum\limits_{\substack{l=1,\\l\neq k}}^{d}(i_l - 1)J_l\quad \text{with}\quad J_l = \prod_{\substack{m=1,\\m\neq k}}^{k-1}n_m.
    \end{equation*}
    
    \item \textit{Mode-i tensor-matrix product.} 
    We denote the mode-$i$ product of a tensor $\bd{\mcal{X}}\in \bb{R}^{n_1\times n_2\cdots\times n_d}$ with a matrix $\mbf{U}\in \bb{R}^{m\times n_i}$ as $\bd{\mcal{X}}\times_i \mbf{U}$. 
    The product $\bd{\mcal{X}}\times_i \mbf{U}\in\bb{R}^{n_1\times \cdots \times n_{i-1}\times m\times n_{i+1}\times \cdots\times n_d}$ is elementwise calculated by
    \begin{equation*}
    (\bd{\mcal{X}}\times_i \mbf{U})_{j_1\cdots j_{i-1} k j_{i+1}\cdots j_d} = \sum_{l = 1}^{n_i} \bd{\mcal{X}}_{j_1\cdots j_{i-1} l j_{i+1}\cdots j_d} \mbf{U}_{k l}.
    \end{equation*}
    It can also be expressed in terms of unfolded tensors as follows
    \begin{equation*}
    \bd{\mcal{Y}} = \bd{\mcal{X}}\times_i \mbf{U}\Longleftrightarrow \scr{M}_i (\bd{\mcal{Y}}) = \mbf{U}\scr{M}_i (\bd{\mcal{X}}).
    \end{equation*}
    
    \item \textit{Inner product and norm.}
     For two tensors $\bd{\mcal{X}},\bd{\mcal{Y}} \in\bb{R}^{n_1\times n_2\times \cdots\times n_d}$, the inner product is calculated as the sum of products of their corresponding entries, that is,
    \begin{equation*}
    \langle\bd{\mcal{X}} ,\bd{\mcal{Y}} \rangle = \sum_{j_1=1}^{n_1}\cdots \sum_{j_d=1}^{n_d} x_{j_1\cdots j_d} y_{j_1\dots j_d}.
    \end{equation*}
    The induced norm for given tensor $\bd{\mcal{Z}}$ is expressed by $\bnm\bd{\mcal{Z}}\bnm_F = \sqrt{\langle\bd{\mcal{Z}}, \bd{\mcal{Z}}\rangle}$.
    \item \textit{Multilinear rank and Tucker decomposition.} The multilinear rank  of a tensor $\bd{\mcal{X}}\in \bb{R}^{n_1\times \dots\times n_d}$ is a length-$d$ vector $\rank_T(\bd{\mcal{X}}) = (r_1,\dots, r_d)$, where
    $r_i = \rank(\scr{M}_i(\bd{\mcal{X}})), i = 1,\dots, d$. If the multilinear rank of $\bd{\mcal{X}}$ is $\mbf{r}$, we denote $\mbf{U}_i\in \bb{R}^{n_i\times r_i}$ as the orthogonal matrix which spans the column space of $\scr{M}_i(\bd{\mcal{X}})$ for $i = 1,\dots, d$, respectively. Then the Tucker decomposition of $\bd{\mcal{X}}$ is 
    \begin{equation*}
    \bd{\mcal{X}} = \bd{\mcal{C}}\times_1 \mbf{U}_1\times_2 \cdots\times_d \mbf{U}_d,
    \end{equation*}
    where $\bd{\mcal{C}}\in \bb{R}^{r_1\times\dots\times r_d}$ is called the core tensor and $\mbf{U}_i, i= 1, \dots, d$ are called the factor matrices.
\end{itemize}

\subsection{Related Algorithms}\label{subsection : related algorithms}

A natural approach to solving \eqref{model: low-rank tensor recovery} is to first perform one step of gradient descent on $\bd{\mcal{X}}^{k}$ concerning the objective function, followed by projecting onto the low-multilinear-rank set $\Theta$. i.e.,
\begin{equation}\label{equ: iht step}
\bd{\mcal{X}}^{k+1} = \scr{P}_{\Theta}(\bd{\mcal{X}}^k - \alpha_k \scr{A}^* (\scr{A}\bd{\mcal{X}}^k - \mbf{y})),
\end{equation}
where $k$ is the iteration counter, $\alpha_k$ is the step size at the $k$-th iteration, and $\scr{P}_{\Theta}$ is the projection operator onto the low-multilinear-rank tensor set $\Theta$. However, unlike in the matrix cases where the projection can be analytically calculated using truncated SVD, there is generally no efficient way to compute the projection on $\Theta$. Hence, in practice, its approximation often replaces $\scr{P}_{\Theta}$. There are two popular projection operators: T-HOSVD $\scr{H}_{\mbf{r}}$ \cite{tucker1966some, de2000multilinear} and ST-HOSVD $\scr{H}_{\mbf{r}}^{\text{ST}}$ \cite{vannieuwenhoven2012new}, both of which satisfy the following quasi-projection property.

\begin{definition}[Quasi-projection property of low-multilinear-rank tensor map $\hat{\scr{P}}_{\Theta}$]

Denote $\scr{P}_{\Theta}$ as the projection to the low-multilinear-rank tensor set $\Theta$, i.e., for any $\bd{\mcal{Z}}\in \bb{R}^{n_1\times\dots,\times n_d}$ and tensor $\hat{\bd{\mcal{Z}}}$ with $\rank_{T}(\bd{\mcal{Z}})\leq \mathbf{r}$, $\bnm\bd{\mcal{Z}} - \scr{P}_\Theta (\bd{\mcal{Z}})\bnm_F \leq \bnm\bd{\mcal{Z}} - \hat{\bd{\mcal{Z}}}\bnm_F$. The map $\hat{\scr{P}}_{\Theta}$ satisfies the quasi-projection property with constant $\delta>0$ if for any $\bd{\mcal{Z}}\in \bb{R}^{n_1\times \dots \times n_d}$
\begin{equation}\label{equ: HOSVD error}
    \bnm\bd{\mcal{Z}} - \hat{\scr{P}}_{\Theta}(\bd{\mcal{Z}})\bnm_F\leq\delta\cdot \bnm\bd{\mcal{Z}} - \scr{P}_{\Theta}(\bd{\mcal{Z}})\bnm_F.
\end{equation}
\end{definition}

The T-HOSVD algorithm naturally extends the truncated SVD of a matrix to a higher-order tensor. The computing procedure of T-HOSVD is presented in Algorithm \ref{alg: T-HOSVD}.

\begin{center}    
% \begin{minipage}{0.95\linewidth}
\begin{algorithm}[H]
\caption{T-HOSVD \cite{tucker1966some, de2000multilinear} }
\begin{algorithmic}
\REQUIRE Tensor $\bd{\mcal{Z}} \in \bb{R}^{n_1 \times n_2 \times \cdots \times n_d}$, truncation $\mathbf{r} = \left(r_1, r_2, \cdots, r_d\right)$.\\
\FOR{$i = 1,\cdots, d$}
\STATE $ \mbf{U}_{i} \leftarrow r_i$ leading left singular vectors of $ \scr{M}_i(\bd{\mcal{Z}})$,
\ENDFOR
\STATE Core tensor $\bd{\mcal{C}} = \bd{\mcal{Z}}\times_1   \mbf{U}_{1}^T \times_2  \mbf{U}_{2}^T\times_3 \cdots \times_d  \mbf{U}_{d}^T$
\ENSURE Approximated projection $\scr{H}_{\mbf{r}}(\bd{\mcal{Z}}) =  \bd{\mcal{C}} \times_1  \mbf{U}_1 \times_2  \mbf{U}_2 \times_3 \cdots \times_d  \mbf{U}_d$.
\end{algorithmic} 
\label{alg: T-HOSVD}
\end{algorithm} 
% \end{minipage}
\end{center}

The ST-HOSVD algorithm employs a Gauss-Seidel type truncation strategy, effectively reducing the computational complexity inherent in the T-HOSVD process.

\begin{center}    
\begin{algorithm}[H]
\caption{Sequentially Truncated HOSVD (ST-HOSVD)\cite{vannieuwenhoven2012new}}
\begin{algorithmic} 
\STATE \textbf{Input:} Tensor $\bd{\mcal{Z}} \in \bb{R}^{n_1 \times n_2 \times \cdots \times n_d}$, truncation $\mathbf{r} = \left(r_1, r_2, \cdots, r_d\right)$.\\
\STATE $\bd{\mcal{B}} = \bd{\mcal{Z}}$.
\FOR{$n \in\left\{1, \dots, d\right\}$}
\STATE $ \mbf{U}_{n},  \mbf{\Sigma}_{n},  \mbf{V}_{n}^T \leftarrow$ truncated-$r_n$ SVD of $ \scr{M}_{n}(\bd{\mcal{B}})$,
\STATE $\bd{\mcal{B}} = \scr{M}_n^{-1}( \mbf{\Sigma}_n  \mbf{V}_n^T)$
\ENDFOR
\STATE Core tensor $\bd{\mcal{C}} = \bd{\mcal{B}}$.
\ENSURE Approximated projection $\scr{H}_{\mbf{r}}^{\text{ST}}(\bd{\mcal{Z}}) = \bd{\mcal{C}}\times_1  \mbf{U}_{1} \times_2  \mbf{U}_{2}\times_3 \cdots \times_d  \mbf{U}_{d}$.
\end{algorithmic} 
\label{alg: ST-HOSVD}
\end{algorithm}
\end{center}

In Algorithm \ref{alg: ST-HOSVD}, the ST-HOSVD method calculates the truncated SVD of the intermediate matrix $\scr{M}_i(\bd{\mcal{B}})$ instead of $\scr{M}_i(\bd{\mcal{Z}})$ to obtain the factor matrix $\mbf{U}_i$. Subsequently, $\bd{\mcal{B}}$ is updated as $\bd{\mcal{B}}\times_i \mbf{U}_i^T$. This process reduces the dimensions of $\bd{\mcal{B}}$ across each mode, thereby diminishing the computational complexity required for the truncated SVD in subsequent modes. Furthermore, it is shown in \cite{xiao2021efficient} that ST-HOSVD can achieve minimal computational cost when the truncated SVD is performed sequentially, starting from modes with lower ranks and progressing towards those with higher ranks.
Regarding the approximation error, both T-HOSVD and ST-HOSVD satisfy the following quasi-projection property with the same approximation constant, as discussed in \cite{vannieuwenhoven2012new}.

\begin{proposition}[Quasi-projection property of T-HOSVD and ST-HOSVD]\label{prop: quasipro}
The T-HOSVD and ST-HOSVD methods in Algorithm \ref{alg: T-HOSVD} and \ref{alg: ST-HOSVD} satisfy the quasi-projection property with approximation constant $\sqrt{d}$.
\end{proposition}

By replacing  $\scr{P}_{\Theta}$ in \eqref{equ: iht step} by T-HOSVD and ST-HOSVD, \cite{rauhut2017low} and \cite{de2017low} proposed the TIHT algorithm and SeMPIHT algorithm. Concerning the selection of the step size $\alpha_k$, \cite{rauhut2017low} proposed two variants of IHT: the Constant step size IHT (CIHT, $\alpha_k = 1$) and the Normalized step size IHT (NIHT). The normalized step size is defined as follows:
\begin{equation}\label{equ: IHT stepsize}
\alpha_k = \argmin_{\alpha}\bnm\scr{A}\left(\bd{\mcal{X}}^k + \alpha \scr{F}^k(\bd{\mcal{G}}^k) - \bd{\mcal{T}}\right) \bnm_F^2 = \frac{\bnm\scr{F}^k(\bd{\mcal{G}}^k)\bnm_F^2}{\bnm\scr{A}(\scr{F}^k(\bd{\mcal{G}}^k))\bnm_F^2}.
\end{equation}
Here, $\bd{\mcal{G}}^k = \scr{A}^* \scr{A}(\bd{\mcal{T}} - \bd{\mcal{X}}^k)$ represents the negative gradient direction, and the operator $\scr{F}^k:\bb{R}^{n_1\times \dots\times n_d}\longrightarrow \bb{R}^{n_1\times \dots \times n_d}$ is defined as
\begin{equation*}
\scr{F}^k(\bd{\mcal{Z}}) := \bd{\mcal{Z}}\times_{i\in [d]} (\mbf{U}^k_i (\mbf{U}^k_i)^T),
\end{equation*}
with $\mbf{U}_i^k$ being the factor matrices of $\bd{\mcal{X}}^k$. Thus, $\scr{F}^k$ projects tensor $\bd{\mcal{Z}}$ onto the subspace of tensors whose column space of mode-$i$ matricization is spanned by $\mbf{U}_i^k$, for $i = 1,\dots, d$.

The RGD algorithm, on the other hand, treats the fixed multilinear rank tensor set as a Riemannian manifold and performs additional tangent space projection before T-HOSVD $\scr{H}_{\mbf{r}}$. Let $\bd{\mcal{X}}^{k} = \bd{\mcal{C}}^k\times_{i\in [d]} \mbf{U}^k_i$ be its multilinear factorization, the tangent space at $\bd{\mcal{X}}^{k}$ is defined as:
\begin{equation}\label{equ: tensor tangent space}
\bb{S}^k: = \left\{\bd{\mcal{Z}}\in \bb{R}^{n_1\times \dots\times n_d} :\bd{\mcal{Z}} = \bd{\mcal{D}}\times_{i\in [d]}\mbf{U}^k_i + \sum_{i = 1}^d \bd{\mcal{C}}^k\times_{j\in [d]\backslash i} \mbf{U}^k_j\times_i \mbf{V}_i \right\},
\end{equation}
here $\bd{\mcal{D}}\in\bb{R}^{r_1\times \dots\times r_d}$, $\mbf{V}_i\in \bb{R}^{n_i\times r_i}$, and the orthogonality condition $(\mbf{U}^k_i)^T\mbf{V}_i=0$ holds for $i = 1,\dots, d$. It is shown in \cite{kressner2014low, cai2020provable} that the tensors within $\bb{S}^k$ possess a maximum multilinear rank of $2\mbf{r}$, thereby it suffices to apply thresholding only on tensors of size $2\mbf{r}$. Nevertheless, the computation of the $d+1$ orthogonal components (namely, $\bd{\mcal{D}}$ and $\mbf{V}_i$ for $i = 1,\dots, d$) delineated in \eqref{equ: tensor tangent space} entails computational complexity comparable to that of $\scr{H}_{\mbf{r}}$, with the coefficient of the highest order term being dependent on the order $d$ \cite{kressner2014low, cai2020provable}.

\section{Algorithm and Main Results}\label{section: Algorithm and Main Results}
\subsection{Single Mode Quasi Riemannian Gradient Descent Algorithm}\label{subsection: SMQRGD algorithm}

In this work, we follow the framework in (\ref{equ: iht step}) and mainly focus on using ST-HOSVD for the projection operator $\scr{P}_{\Theta}$. In light of the methodology in RGD, our primary goal is to import the tangent space projection technique for the first mode truncated SVD of ST-HOSVD to achieve better computational efficiency. The details of our approach are outlined in Algorithm \ref{alg: SMQRGD}.
\begin{center} 
% \begin{minipage}{0.95\linewidth}
\begin{algorithm}[H]
\caption{Single Mode Quasi Riemannian Gradient Descent Algorithm}
\begin{algorithmic} 
\STATE\textbf{Initialization: } $\bd{\mcal{X}}^0 = \scr{H}^{\text{ST}}_{\mathbf{r}}(\scr{A}^*\scr{A}(\bd{\mcal{T}})), \bd{\hat{\mcal{X}}}^0 = \scr{H}_{r_{1}}(\scr{A}^*\scr{A}(\bd{\mcal{T}}))$. 
\FOR{$k = 0, 1, 2 \dots$}
\STATE $\bd{\mcal{G}}^k = \scr{A}^* \scr{A}(\bd{\mcal{T}} - \bd{\mcal{X}}^k)$\\
% \STATE $\alpha_k$ is chosen from \eqref{equ: step size choice}
\STATE $\bd{\mcal{W}}^{k+1} = \scr{P}_{\hat{\bb{T}}^k}(\bd{\mcal{X}}^{k} + \alpha_{k}\bd{\mcal{G}}^k)$. \\
\STATE $\bd{\mcal{X}}^{k+1} = \scr{H}^{\text{ST}}_{\mathbf{r}}(\bd{\mcal{W}}^{k+1})$, $\bd{\hat{\mcal{X}}}^{k+1} = \scr{H}_{r_1}(\bd{\mcal{W}}^{k+1})$\\
\ENDFOR
\end{algorithmic} 
\label{alg: SMQRGD}
\end{algorithm}
% \end{minipage}
\end{center}

Analogous to SeMPIHT, SM-QRGD uses ST-HOSVD $\scr{H}^{\text{ST}}_{\mathbf{r}}(\cdot)$ as the projection operator onto the low-rank tensor set $\Theta$. Distinct from SeMPIHT, SM-QRGD performs additional tangent space projection before the operation $\scr{H}^{\text{ST}}_{\mathbf{r}}(\cdot)$. The tangent space projection operator $\scr{P}_{\hat{\bb{T}}^k}$ is defined as 
\begin{equation*}
\scr{P}_{\hat{\bb{T}}^k} := \scr{M}_1^{-1}\circ\scr{P}_{\hat{T}^k}\circ \scr{M}_1,
\end{equation*}
where $\hat{T}^k$ is the sum of column space and row space of $\scr{M}_1(\bd{\hat{\mcal{X}}}^{k})$, defined by
\begin{equation}\label{equ: low-rank tangent space}
    \hat{T}^k := \left\{\mbf{Z} = \mbf{U}_1^k\mbf{R}^T + \mbf{L} (\mbf{V}_1^k)^T\text{ with } \mbf{L}\in \bb{R}^{n_1\times r_1}, \mbf{R}\in \bb{R}^{\prod\limits_{i = 2}^{d}n_i \times r_1}\right\},
\end{equation}
where $\mbf{U}_1^k$ and $\mbf{V}_1^k$ are left and right singular matrices of $\scr{M}_1(\bd{\hat{\mcal{X}}}^{k})$, which can be directly obtained from the ST-HOSVD process in the previous iteration. Then, following \cite{wei2016guarantees}, the projection operator of $\scr{P}_{\hat{T}^k}$ for a given $\mbf{Y}$ is calculated as
\begin{equation}\label{equ: tangent space projection}
    \scr{P}_{\hat{T}^k}(\mbf{Y}) = \mbf{U}_1^k(\mbf{U}_1^k)^T \mbf{Y} + \mbf{Y} \mbf{V}_1^k (\mbf{V}_1^k)^T - \mbf{U}_1^k (\mbf{U}_1^k)^T \mbf{Y} \mbf{V}_1^k (\mbf{V}_1^k)^T.
\end{equation}

We give the following conceptual illustration of how the additional tangent space projection can be advantageous in reducing the computational cost. The detailed complexity analysis is deferred to Appendix  \ref{Appendix: Computational complexity}. We note that in \eqref{equ: low-rank tangent space}  matrices within $\hat{T}^k$ possess a rank at most $2r_1$, which implies that the projection $\scr{P}_{\hat{T}^k}(\mbf{Y})$ can be expressed as $\mbf{A}\mbf{B}^T$, where $\mbf{A}$ and $\mbf{B}$ represent matrices of dimensions ${n_1\times 2r_1}$ and ${(\prod_{i = 2}^d n_i)\times 2r_1}$. Implementing QR decompositions on $\mbf{A}$ and $\mbf{B}$ results in $\mbf{A} = \mbf{Q}_1 \mbf{R}_1$ and $\mbf{B} = \mbf{Q}_2 \mbf{R}_2$, leading to $\scr{P}_{\hat{T}^k}(\mbf{Y}) = \mbf{Q}_1 \mbf{R}_1 \mbf{R}_2^T \mbf{Q}_2^T$. It is observed that $\mbf{Q}_1$ and $\mbf{Q}_2$ are both orthogonal matrices with dimensions ${n_1\times 2r_1}$ and ${(\prod_{i=2}^d n_i)\times 2r_1}$.
% while $\mbf{R}_1$ and $\mbf{R}_2$ are both $2r_1\times 2r_1$ matrices. 
This orthogonality enables the computation of truncated-$r_1$ SVD of $\mbf{Y}$ to a smaller $2r_1\times 2r_1$ matrix $\mbf{R}_1 \mbf{R}_2^T$, which can effectively reduce the computational cost of the first mode SVD calculations in ST-HOSVD. 

While the single mode tangent space projection cannot be generally treated as the classical Riemannian gradient descent on a low-rank tensor due to the mismatch between $\hat{T}^k$ and the tangent space of $\scr{M}_1(\bd{\mcal{X}}^k)$, the following Lemma \ref{lemma: lmPSTS} indicates that $\scr{M}_1(\bd{\mcal{X}}^k)$ actually belongs to  $\hat{T}^k$. 
\begin{lemma}[Projection onto Substitution Tangent Space]\label{lemma: lmPSTS}
The projection of $\bd{\mcal{X}}^k$ onto the mode-1 tangent space of $\bd{\hat{\mcal{X}}}^k$ is still $\bd{\mcal{X}}^k$, e.g., $\scr{P}_{\hat{\bb{T}}^k}(\bd{\mcal{X}}^k) = \bd{\mcal{X}}^k$.
\end{lemma}
\begin{proof}
As $\bd{\mcal{X}}^k = \scr{H}^{\text{ST}}_{\mathbf{r}} (\bd{\mcal{W}}^k) = \bd{\mcal{W}}^k\times_{i\in [d]} \mbf{U}_i^k (\mbf{U}_i^k)^T$, then
\begin{equation*}\begin{aligned}
\scr{M}_1(\bd{\mcal{X}}^k)& = \mbf{U}_1^k (\mbf{U}_1^k)^T \scr{M}_1(\bd{\mcal{W}}^k)\left(\mbf{U}_d^k (\mbf{U}_d^k)^T \otimes \cdots \otimes \mbf{U}_2^k (\mbf{U}_2^k)^T\right)\\
& = \mbf{U}_1^k \mbf{\Sigma}_1 (\mbf{V}_1^k)^T \left(\mbf{U}_d^k (\mbf{U}_d^k)^T \otimes \cdots \otimes \mbf{U}_2^k (\mbf{U}_2^k)^T\right),
\end{aligned}
\end{equation*}
It can be obtained that $\mbf{U}_1^k (\mbf{U}_1^k)^T\scr{M}_1(\bd{\mcal{X}}^k) = \scr{M}_1(\bd{\mcal{X}}^k)$. By the tangent space formula in \eqref{equ: tangent space projection},
\begin{equation*}
\begin{aligned}
\scr{P}_{\hat{T}^k}(\scr{M}_1(\bd{\mcal{X}}^k)) &= \scr{M}_1(\bd{\mcal{X}}^k) + \scr{M}_1(\bd{\mcal{X}}^k) \mbf{V}_1^k (\mbf{V}_1^k)^T - \scr{M}_1(\bd{\mcal{X}}^k) \mbf{V}_1^k (\mbf{V}_1^k)^T\\
& = \scr{M}_1(\bd{\mcal{X}}^k),
\end{aligned}
\end{equation*}
Thus, $\scr{P}_{\hat{\bb{T}}^k}(\bd{\mcal{X}}^k) = \bd{\mcal{X}}^k$.
\end{proof}

Therefore, Algorithm \ref{alg: SMQRGD} can be characterized as a quasi-type of Riemannian gradient descent on the chosen single mode, as illustrated in Figure \ref{fig: Illustration}. 

For the step size, we can use the following constant step size or the normalized step size,  similar to the approach used in IHT.
\begin{equation}\label{equ: step size choice}
    \begin{aligned}
    \textbf{Constant step size: } &\ \ \alpha_k = 1 ,\\
    \textbf{Normalized step size: } &\ \ \alpha_k = \frac{\bnm\scr{P}_{\hat{\bb{T}}^k}(\bd{\mcal{G}}^k)\bnm_F^2}{\bnm\scr{A}(\scr{P}_{\hat{\bb{T}}^k}(\bd{\mcal{G}}^k))\bnm_F^2}.
    \end{aligned}
\end{equation}

With Lemma \ref{lemma: lmPSTS} and by the inspection of \eqref{equ: IHT stepsize}, it can be seen that the normalized step size is also the steepest as $\scr{P}_{\hat{\bb{T}}^k}(\bd{\mcal{X}}^{k} + \alpha_{k}\bd{\mcal{G}}^k)= \bd{\mcal{X}}^{k} + \alpha_{k}\scr{P}_{\hat{\bb{T}}^k}(\bd{\mcal{G}}^k)$.

\begin{figure}[H]
\centering
\begin{minipage}[t]{0.48\linewidth}
		\centering
		\includegraphics[width=.95\linewidth]{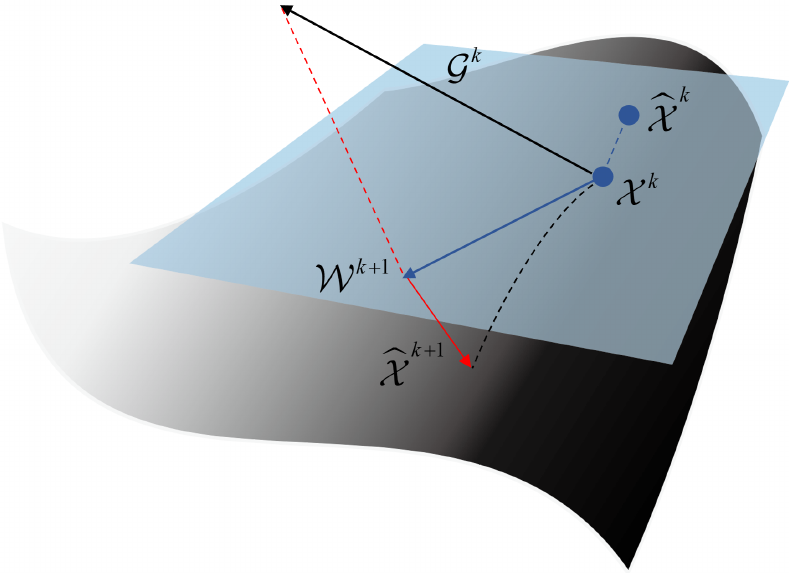}
\end{minipage}
\caption{Illustration for the SM-QRGD Algorithm. }
\label{fig: Illustration}
\end{figure}

\subsection{Complexity Analysis and Comparison}\label{subsection: complexity}

In this section, we compare the computational complexity of the proposed method SM-QRGD with  NIHT, SeMPIHT, and RGD (the detailed derivations of the results can be found in Appendix \ref{Appendix: Computational complexity}). For ease of exposition, we assume that $n_i = n, r_i = r, i = 1, \dots, d$, and the results are reported in Table \ref{tab: Comparison in complexity}.
\begin{table}[H]
\small
    \centering
    \begin{tabular}{ccc}
        \hline
        \hline
        Algorithm & Thresholding Operation &  Step Size Normalization\\
        \hline
        \textcolor{black}{NIHT} \cite{rauhut2017low} & $ O(d n^d r)$ & $2 n^d r + o(n^d r) + O(n^d)$ \\
        \textcolor{black}{SeMPIHT} \cite{de2017low} & $ O(n^d r)$ & $2 n^d r + o(n^d r) + O(n^d)$ \\
        RGD \cite{cai2020provable} & $(d+2) n^d r+ o(n^d r)$ & $(d+1) n^d r + o(n^d r) + O(n^d)$ \\
        SM-QRGD \textbf{(this paper)} & $3 n^d r + o(n^d r)$ & $4 n^d r + o(n^d r) + O(n^d)$ \\
        \hline 
    \end{tabular}
    \caption{The computation complexity for thresholding operation and step size normalization of four algorithms in each iteration.}
    \label{tab: Comparison in complexity}
\end{table}
Table \ref{tab: Comparison in complexity} demonstrates that the primary terms of the four algorithms are identical, with differences in the coefficients in front of them. Compared to RGD, the complexity coefficient for computing the thresholding operator in SM-QRGD is constant and usually smaller. Also, the SM-QRGD method has a lower step size normalization computational cost than the RGD method as $d\geq 3$. Compared to SeMPIHT, although SM-QRGD has a similar computational cost to SeMPIHT, SM-QRGD can be practically more efficient due to the difference in the first mode calculation during the ST-HOSVD process. Specifically, SeMPIHT requires the computation of a truncated-$r$ SVD of a non-structured $n\times n^{d-1}$ matrix, generally incurring a computational cost of $O(n^d r)$ flops, with a big coefficient in front of the $n^d r$ depending on the chosen SVD algorithm. Conversely, for SM-QRGD, the most costly part arises from matrix products within the tangent space projection, which only requires $3 n^d r$ flops and can be further accelerated by parallel computing. Therefore, SM-QRGD can be practically more appealing than SeMPIHT, as evidenced in our numerical simulations in Section \ref{section: numerical experiments}.

\subsection{Convergence and Recovery Guarantee}
As in \cite{rauhut2017low, de2017low, cai2020provable}, the convergence analysis of tensor recovery algorithms relies on the restricted isometry property (RIP) condition of the operator $\scr{A}$. SM-QRGD requires the operator $\scr{A}$ to have the following first-mode RIP condition.
\begin{definition}[Tensor First-mode Restricted Isometry Property]
\label{def: RIP property}
For the linear operator $\scr{A}: \bb{R}^{n_1\times\dots\times n_d}\longmapsto \bb{R}^{m}$, an operator $\scr{A}$ is said to satisfy the first-mode RIP with rank $r_1$, if there exists a constant $R_{r_1}$ such that for all $\bd{\mcal{Z}}\in \bb{R}^{n_1\times\dots\times n_d}$ with $\rank (\scr{M}_1(\bd{\mcal{Z}}))\leq r_1$, the following inequality holds
\begin{equation}\label{equ: mode-1 RIP condition}
(1- R_{r_1})\bnm\bd{\mcal{Z}}\bnm_F^2\leq \bnm\scr{A}(\bd{\mcal{Z}})\bnm_F^2\leq (1+R_{r_1})\bnm\bd{\mcal{Z}}\bnm_F^2.
\end{equation}
The constant $R_{r_1}$ is called the first-mode restricted isometry constant (1-RIC).
\end{definition}
\begin{remark}
    The first-mode RIP can be regarded as a particular case of TRIP with multilinear rank $\bar{\mathbf{r}} = (r_1, n_2,\dots, n_d)$. Suppose $n_i = n, r_i = r, i = 1,\dots, d$ and $\mbf{y}$ are i.i.d subgaussian measurements, Theorem 2 in \cite{rauhut2017low} shows that \eqref{equ: mode-1 RIP condition}  holds with high probability provided that the number $m\geq R_{r}^{-2} O(r n^{d-1}+d n^2)$.
\end{remark}

\begin{theorem}[Recovery guarantee with $\alpha = 1$]\label{thm: alpha =1}
Assume $\scr{A}$ satisfies Definition \ref{def: RIP property}.
% which is analyzed in \cite{rauhut2017low}. 
Define the following constant:
\begin{equation*}
     \gamma_1 :=  R_{3 r_1}\left(8\sqrt{r_1}\kappa_1\left((\sqrt{d-1}+1) (R_{3 r_1}+2) + R_{3 r_1}\right) + \sqrt{d} + 3\right),
\end{equation*}
where $\kappa_1 := \sigma_1(\scr{M}_1(\mcal{T}))/\sigma_{r_1}(\scr{M}_1(\mcal{T}))$.
In particular, $\gamma_1 < 1$ can be satisfied if
\begin{equation*}
R_{3 r_1} < \min \left(\frac{1}{2}, \frac{1}{ (20\sqrt{d-1} + 24)\sqrt{r_1}\kappa_1 + \sqrt{d} + 3} \right),
\end{equation*}
then the iterates of the SM-QRGD algorithm with constant step size $\alpha = 1$ and the initial point $\bd{\mcal{X}}^0 = \scr{H}^{ST}_{\mathbf{r}}(\scr{A}^*\scr{A}(\bd{\mcal{T}}))$ satisfy
\begin{equation*}
\bnm\bd{\mcal{X}}^{k+1} - \bd{\mcal{T}}\bnm_F\leq (\sqrt{d}+1)(\frac{5}{2} \sqrt{d-1} + 3) \gamma_1^{k}\bnm\bd{\mcal{T}}\bnm_F.
\end{equation*}

\end{theorem}

The convergence results of the SM-QRGD algorithm with normalized step size are given as follows.
\begin{theorem}[Recovery guarantee with normalized step size]
With the same assumptions and notations in Theorem \ref{thm: alpha =1}. Define the following constant:
\begin{equation*}
\gamma_2 := \frac{2R_{3 r_1}}{1-R_{3 r_1}}\left(4 \sqrt{r_1}\kappa_1 \left(2 (\sqrt{d-1}+1) + R_{3 r_1}\right) + \sqrt{d}+2 \right),
\end{equation*}
where $\kappa_1 := \sigma_1(\scr{M}_1(\mcal{T}))/\sigma_{r_1}(\scr{M}_1(\mcal{T}))$.    
In particular, $\gamma_2< 1$ can be satisfied if
\begin{equation*}
R_{3 r_1} < \min \left(\frac{1}{2}, \frac{1}{\left(32 \sqrt{d-1} + 40\right)\sqrt{r_1}\kappa_1  + 4 \sqrt{d}+ 8} \right),
\end{equation*}
then the iterates of the SM-QRGD algorithm with normalized step size \eqref{equ: step size choice} and the initial point $\bd{\mcal{X}}^0 = \scr{H}^{ST}_{\mathbf{r}}(\scr{A}^*\scr{A}(\bd{\mcal{T}}))$ satisfy
\begin{equation*}
\bnm\bd{\mcal{X}}^{k+1} - \bd{\mcal{T}}\bnm_F
 \leq (\sqrt{d}+1)(8 \sqrt{d-1}+10)\gamma_2^{k} \bnm\bd{\mcal{T}}\bnm_F.
\end{equation*}

\label{thm: SMQRGD}
\end{theorem}
%%%%%%%%%%%%%%%%%%%%%%%%%%%%%%%%%%%%%%%%%%
\begin{remark}
    It is observed in Theorem \ref{thm: alpha =1} and \ref{thm: SMQRGD} that we need to ensure the 1-RIC is sufficiently small for convergence. This condition, however, can be met with high probability when the number of measurements $m\geq O(d r^2 n^{d-1} \kappa_1^2)$ (assume $n_i = n, r_i = r, i = 1,\dots, d$.)
\end{remark}

\section{Numerical Experiments}\label{section: numerical experiments}
In this section, we evaluate the proposed algorithm for solving the tensor completion problem, i.e., the operator $\scr{A}$ in (\ref{model: low-rank tensor recovery}) is projection operator $\scr{P}_{\Omega}$. We mainly focus on a cubic tensor of dimension
$\mbf{n} = (n, n, n)$ with multilinear rank $\mathbf{r} = (r, r, r), r \leq n$. To initiate the process, we generate a random tensor by independently sampling its entries from a standard normal distribution. Subsequently, we apply T-HOSVD to transform this random tensor into a low-multilinear-rank tensor denoted as $\bd{\mcal{T}}$.
The measurement $\mbf{y}$ is calculated by $\mbf{y} = \scr{P}_{\Omega}(\bd{\mcal{T}})$, where $\Omega$ is randomly sampled from the indices of $\bd{\mcal{T}}$ with a sampling ratio $\rho = |\Omega|/n^3$. We use the following relative error under the Frobenius norm as the metric to measure the recovery quality:
\begin{equation*}
\mathbf{err}^k := \frac{\bnm\bd{\mcal{X}}^k - \bd{\mcal{T}}\bnm_F}{\bnm\bd{\mcal{T}}\bnm_F}.
\end{equation*}

\begin{figure}[htp]
\centering
	\subfigure[Phase Transition]{\includegraphics[width=.48\linewidth]{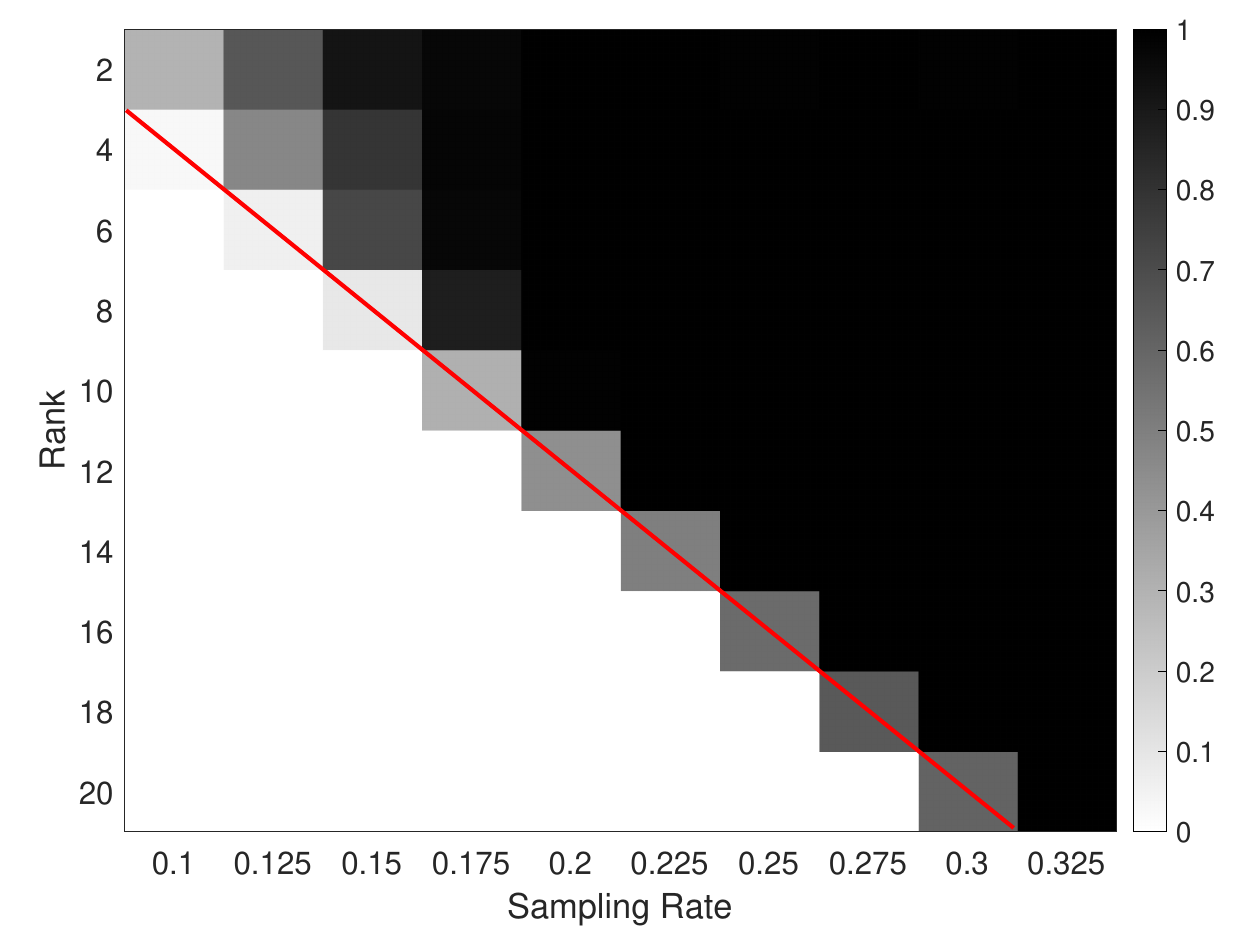}
	\label{subfig: phase transition}}
	\subfigure[Noisy Data Recovery]{
	\includegraphics[width=.48\linewidth]{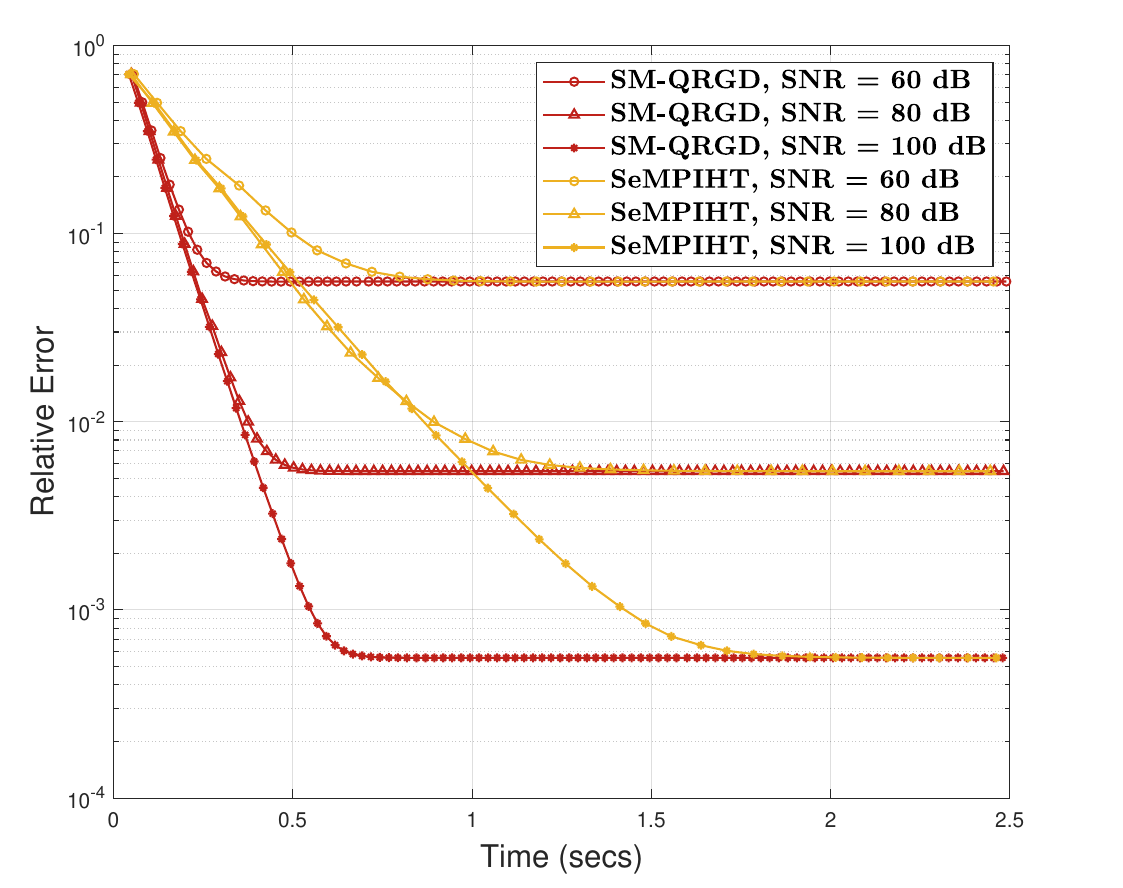}
	\label{subfig: noisy data recovery}}
	\caption{(a): The phase transition plot for varying rank $r$ and sampling rate $\rho$ when $n = 100$. (b): The CPU times of SM-QRGD and SeMPIHT v.s. relative error under different signal-to-noise ratios when $n = 100, r = 3$ and $\rho = 0.3$.}
	\label{fig: Phase transition and Robust}
\end{figure}

\textbf{Phase transition of SM-QRGD.} 
We set the tensor size to $n = 100$ and vary the rank $r$ and the sampling rate $\rho$. For every combination of rank and sampling rate, 100 random tests are performed to determine the success rate. A test is deemed as a successful recovery if the relative error between the reconstructed tensor and $\bd{\mcal{T}}$   is less than or equal to $10^{-5}$. The phase transition is depicted in Fig. \ref{subfig: phase transition}. These results indicate a linear relationship between the sampling complexity $\rho$ necessary for successful recovery and the rank $r$. The results are aligned with prevailing findings concerning the required sampling complexity for successful tensor recovery.

\textbf{Robustness of SM-QRGD.} We evaluate the robustness of the SM-QRGD algorithm in scenarios where the observed data is contaminated by additive noise.
The noisy observations are given by $\scr{P}_{\Omega}(\bd{\mcal{T}} + \bd{\mcal{J}})$, where $\bd{\mcal{J}}(i_1, i_2, i_3)$ follows a normal distribution $\mcal{N}(0, \sigma_{\bd{\mcal{J}}}^2)$. The noise level is set following the approach used in \cite{tong2022scaling}. 

The comparison of SM-QRGD and SeMPIHT using constant stepsize $\alpha_k = 1$ is presented in Fig. \ref{subfig: noisy data recovery}. Three distinct noise levels, i.e., $\text{SNR} = 60, 80, 100\ \text{dB}$, are considered (The Signal-to-Noise Ratio (SNR) in dB are defined as $\text{SNR} := 10\log_{10}\frac{\|\bd{\mcal{T}}\|_F^2}{n^3\sigma_{\bd{\mcal{J}}}^2}$). Our SM-QRGD method reaches the same relative error as the SeMPIHT method but exhibits a notably faster convergence speed. 
\begin{figure}[htp]
	\begin{minipage}[t]{0.49\linewidth}
		\centering
		\subfigure[Iteration v.s. Relative Error]{
		\includegraphics[width=.96\linewidth]{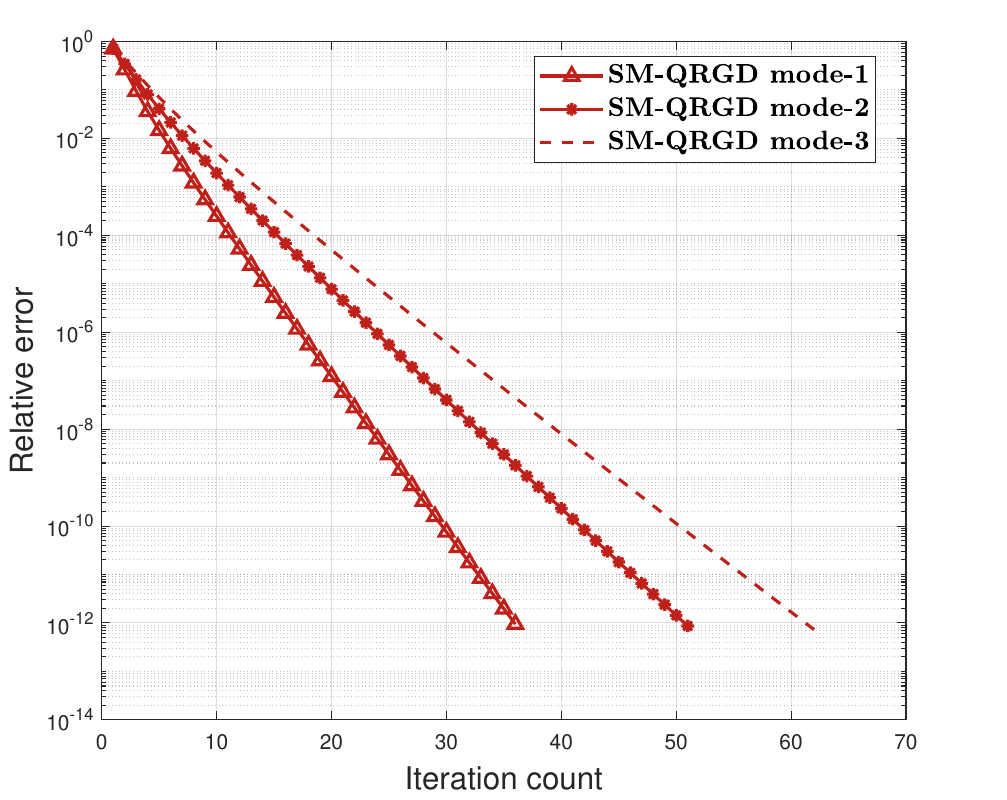}
		\label{subfig: diff rank iter vs error}}
	\end{minipage}
	\begin{minipage}[t]{0.49\linewidth}
		\centering
		\subfigure[Runtime v.s. Relative Error]{
		\includegraphics[width=.96\linewidth]{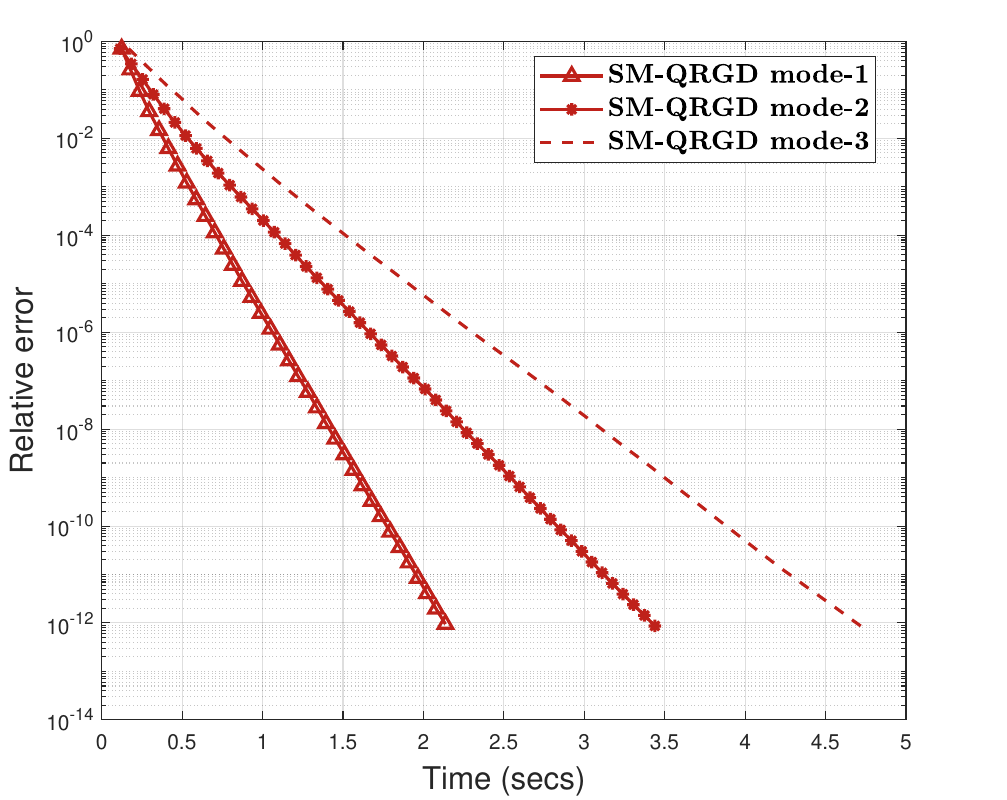}
		\label{subfig: diff rank CPU time vs error}}
	\end{minipage}
	\caption{Results for different modes tangent space projection with $n = 100, r_1 = 10, r_2 = 20, r_3 = 30$ and $\rho = 0.3$.}
	\label{fig: Different mode projection}
\end{figure}

\begin{figure}[htp]
	\begin{minipage}[t]{0.49\linewidth}
		\centering	
		\subfigure[Iteration v.s. Relative Error]{\includegraphics[width=.96\linewidth]{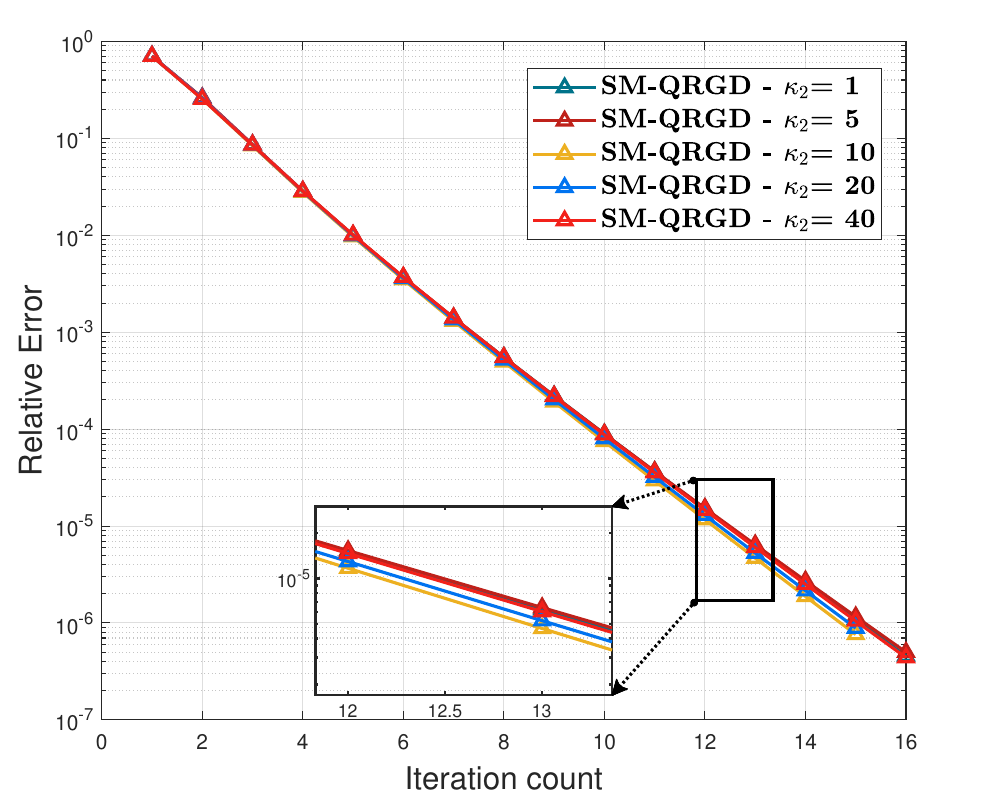}
		\label{subfig: diff cond num Iter vs error}}
	\end{minipage}
	\begin{minipage}[t]{0.49\linewidth}
		\centering
		\subfigure[Runtime v.s. Relative Error]{\includegraphics[width=.96\linewidth]{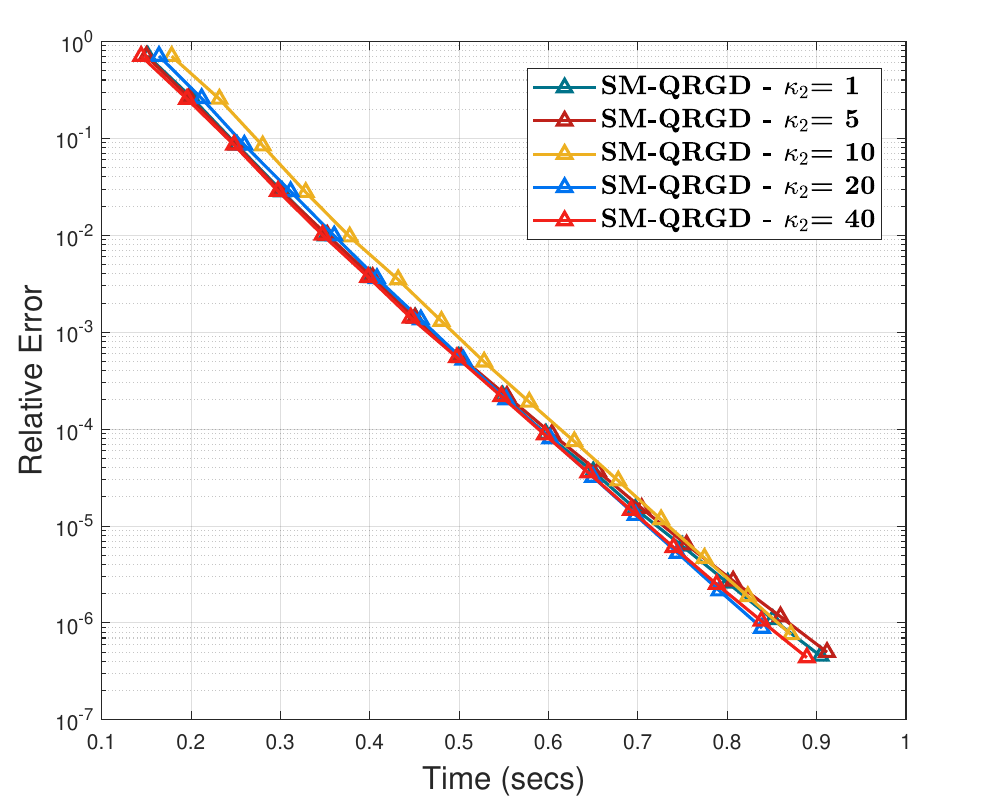}
		\label{subfig: diff cond num CPU time vs error}}
	\end{minipage}
	\caption{Results of SM-QRGD under different condition number}
	\label{fig: Different condition number}
\end{figure}

\textbf{Tangent space projection on different modes.} 
As discussed, the different mode selections in SM-QRGD can result in different performances.
Here, to assess its influence,  we consider a tensor $\bd{\mcal{T}}$ of size $\bb{R}^{n\times n\times n}$ with distinct ranks for each mode, i.e., $\mathbf{r} = (r_1, r_2, r_3)$ ($r_1< r_2< r_3$). With the same sampling operator $\scr{P}_\Omega$, we run $3$ different experiments by separately choosing modes $1$,$2$ and $3$ as tangent space to project, i.e., for mode $1$ (or $2, 3$) tangent space projection, the ordering for mode matrix SVD is $[1,2,3]$\ (or $[2,1,3],\ [3,1,2]$). The relative error \textit{v.s.} iteration counter and CPU time are presented in Fig. \ref{subfig: diff rank iter vs error} and \ref{subfig: diff rank CPU time vs error}, respectively. It can be observed that selecting a mode with a smaller rank can yield a faster convergence, which is consistent with our theoretical analysis (Theorem \ref{thm: alpha =1}, \ref{thm: SMQRGD}).

\textbf{SM-QRGD on tensor with different condition numbers.}
To assess the performance of SM-QRGD under different levels of ill-posedness, we experiment by varying the condition number of the tensor $\bd{\mcal{T}}$ (The condition number of a given mode is defined as the condition number of the corresponding matrixization). To simplify the presentation, we only consider the condition number of the second mode $\kappa_2$  while setting the other two modes $\kappa_1,\kappa_3$ fixed. As in \cite{tong2022scaling}, we generate a core tensor $\bd{\mcal{S}}\in \bb{R}^{r\times r\times r}$ by the formula  $\bd{\mcal{S}}(j_1, j_2, j_3) = \sigma_{j_2}/\sqrt{r}$ if $j_1 +j_2 + j_3 \equiv 0\ (\operatorname{mod} r)$. where $\sigma_{j_2}, j_2 = 1,\dots, r$ take values equispaced from $1$ to $1/\kappa_2$. Then, it can be verified that the condition numbers of the first and the third mode are equal to 1, and the magnitude of the second mode's condition number equals $\kappa_2$. The results are depicted in Fig. \ref{fig: Different condition number}. The SM-QRGD algorithm maintains stable performance across various sets of condition numbers.

\textbf{Comparison with other algorithms.} We compare our SM-QRGD algorithm with the SeMPIHT and RGD algorithms for solving the tensor completion problem. The three algorithms are compared with and without step size normalization. The test algorithms are terminated when the relative error $\mathbf{err}^k$ comes within the threshold $tol$ or the maximum iteration number $100$ is achieved. Results are presented in Fig. \ref{fig: Tensor completion} and  summarized as follows:

\begin{itemize}
\item Fig \ref{subfig: a} fixes $n = 300,\ r = 5,\ tol = 10^{-10},\  \rho = 0.3$ and plots the relative error \textit{v.s.} the CPU time of those three algorithms with constant and normalized step sizes. It can be observed that, for both choices of step size, SM-QRGD exhibits the fastest convergence.

\item Fig \ref{subfig: b} and Fig \ref{subfig: d} fix $n = 200, tol = 10^{-5}$ and plot the CPU time \textit{v.s.} the rank $r$ and sampling ratio $\rho$, respectively. It is shown that SM-QRGD and RGD are less sensitive to the sampling ratio $\rho$ and rank $r$, while SM-QRGD achieves the desired accuracy more rapidly.

\item Fig \ref{subfig: c} fixes $r = 5, tol = 10^{-5}, \rho = 0.5$ and varies the dimension $n$. The CPU time of both SM-QRGD and RGD remains relatively consistent as the dimension $n$ increases, and SM-QRGD stands out for slightly faster convergence compared to RGD.
\end{itemize}

\begin{figure}[H]
	\begin{minipage}[t]{0.49\linewidth}
		\centering
		\subfigure[Runtime v.s. Relative Error]{\includegraphics[width=.98\linewidth]{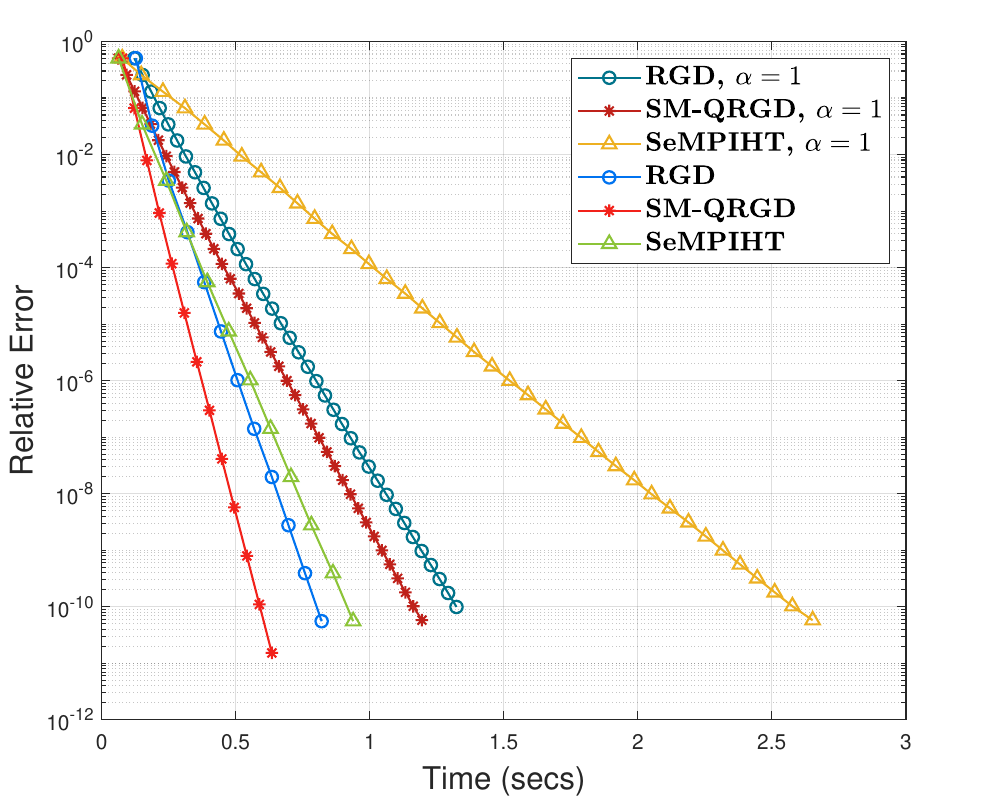}
		\label{subfig: a}}
	\end{minipage}
	\begin{minipage}[t]{0.49\linewidth}
		\centering
		\subfigure[Varying Rank v.s. Runtime]{\includegraphics[width=.98\linewidth]{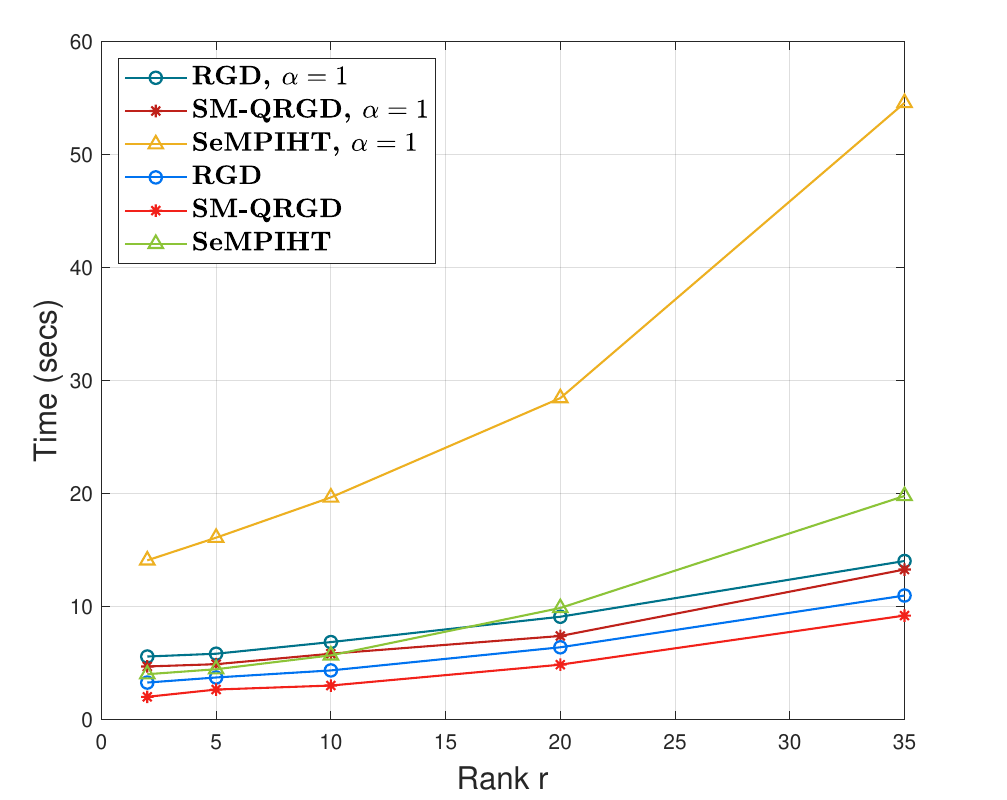}
		\label{subfig: b}}
	\end{minipage}
	
	\begin{minipage}[t]{0.49\linewidth}
		\centering
		\subfigure[Varying Dimension v.s. Runtime]{\includegraphics[width=.98\linewidth]{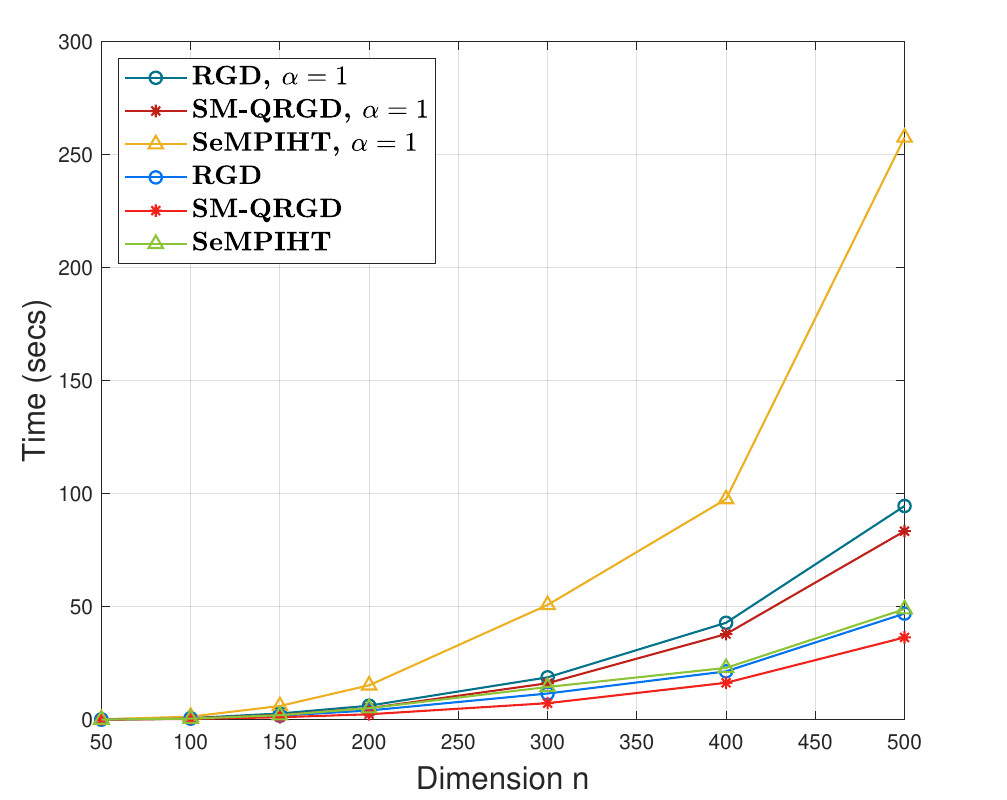}
		\label{subfig: c}}
	\end{minipage}
	\begin{minipage}[t]{0.49\linewidth}
		\centering
		\subfigure[Varying Sampling Ratio v.s. Runtime]{\includegraphics[width=.98\linewidth]{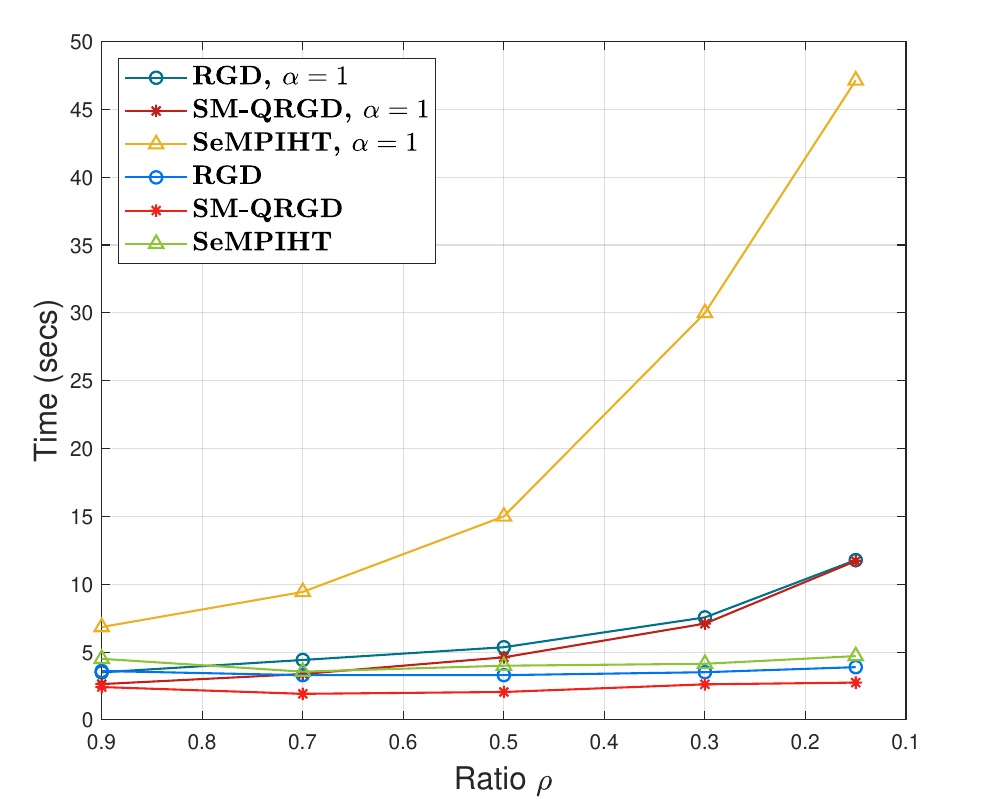}
		\label{subfig: d}}
	\end{minipage}
	\caption{Comparisons with other algorithms}
    \label{fig: Tensor completion}
\end{figure}

\section{Conclusion}\label{section: Discussion and Future Direction}
In this study, we introduced SM-QRGD, a novel algorithm designed for tensor recovery. The approach integrates a low-rank matrix tangent space projection with ST-HOSVD operations for the low-rank tensor set mapping, significantly reducing the computational complexity compared to the state-of-the-art algorithm. Building upon the TRIP assumption, we have established the convergence theory and introduced a recovery guarantee. Numerical results have demonstrated the superior performance of SM-QRGD over current methods.
% There are several research directions for future work. Firstly, while we concentrate on the tensor recovery problem in this paper, the potential of adapting SM-QRGD to other contexts, like robust low-rank tensor recovery, is worth exploring. Secondly, our numerical experiments and convergence analysis suggest that SM-QRGD's performance relies primarily on conditions of the first mode.  So it is interesting to investigate physical problems for tensor recovery problems with specific structures on certain modes. Finally, our convergence theory is based on the TRIP assumption, more precisely the RIP assumption on the first mode matricization of tensor. However, the TRIP assumption makes the sampling complexity of SM-QRGD for recovery guarantee much higher than the current benchmarks. So, it is interesting to establish the theoretical guarantee for SM-QRGD with other techniques.

\section{Appendix}\label{section: Proof}
\subsection{Computational Complexity Analysis}\label{Appendix: Computational complexity}
The complexity analysis is given as follows:
\begin{itemize}
    \item \textbf{Computation of}    $\scr{H}_{\mbf{r}}^{\text{ST}}\circ \scr{P}_{\hat{\bb{T}}}$.
    
    Denote $\scr{M}_1 (\bd{\hat{\mcal{X}}}) = \mbf{\hat{U}} \mbf{\hat{\Sigma}} \mbf{\hat{V}}^T$ where $\scr{M}_1 (\bd{\hat{\mcal{X}}})\in \bb{R}^{n\times n^{d-1}}, \mbf{\hat{U}}\in \bb{R}^{n\times r}$, $\mbf{\hat{\Sigma}}\in \bb{R}^{r\times r}$ and $\mbf{\hat{V}}\in \bb{R}^{n^{d-1}\times r}$. For given $d$-order tensor $\bd{\mcal{Z}}$, 
    \begin{equation*}
    \begin{aligned}
    \scr{P}_{\hat{T}}(\scr{M}_1(\bd{\mcal{Z}})) &= \mbf{\hat{U}} \mbf{\hat{U}}^T \scr{M}_1(\bd{\mcal{Z}}) + \scr{M}_1(\bd{\mcal{Z}}) \mbf{\hat{V}} \mbf{\hat{V}}^T - \mbf{\hat{U}} \mbf{\hat{U}}^T \scr{M}_1(\bd{\mcal{Z}}) \mbf{\hat{V}} \mbf{\hat{V}}^T\\
    & = \mbf{\hat{U}} \mbf{\hat{U}}^T \scr{M}_1(\bd{\mcal{Z}}) (\mbf{I}- \mbf{\hat{V}} \mbf{\hat{V}}^T) + (\mbf{I} - \mbf{\hat{U}} \mbf{\hat{U}}^T) \scr{M}_1(\bd{\mcal{Z}}) \mbf{\hat{V}} \mbf{\hat{V}}^T\\
    &\quad + \mbf{\hat{U}} \mbf{\hat{U}}^T \scr{M}_1(\bd{\mcal{Z}}) \mbf{\hat{V}} \mbf{\hat{V}}^T\\
    & := \mbf{\hat{U}} \mbf{Y}_2^T + \mbf{Y}_1 \mbf{\hat{V}}^T + \mbf{\hat{U}} \mbf{\hat{U}}^T \scr{M}_1(\bd{\mcal{Z}}) \mbf{\hat{V}} \mbf{\hat{V}}^T,
    \end{aligned}
\end{equation*}
    
  where $\mbf{Y}_1 = (\mbf{I} - \mbf{\hat{U}} \mbf{\hat{U}}^T) \scr{M}_1(\bd{\mcal{Z}}) \mbf{\hat{V}}\in \bb{R}^{n\times r}$ and $\mbf{Y}_2 = (\mbf{I}-\mbf{\hat{V}}\mbf{\hat{V}}^T)\scr{M}_1(\bd{\mcal{Z}})^T \mbf{\hat{U}}\in \bb{R}^{n^{d-1}\times r}$. Let $\mbf{Y}_1 = \mbf{Q}_1 \mbf{R}_1$ and $\mbf{Y}_2 = \mbf{Q}_2 \mbf{R}_2$ be the QR decompositions of $\mbf{Y}_1$ and $\mbf{Y}_2$, respectively. From these decompositions, we have $\mbf{\hat{U}}^T \mbf{Q}_1 = 0, \mbf{\hat{V}}^T \mbf{Q}_2 = 0$. The computation of $\mbf{Y}_1, \mbf{Y}_2$ requires $2(n^2 r+ n^d r + n^{d-1} r^2)$ flops and performing QR decompositions costs $O(n r^2 + n^{d-1} r^2)$ flops. Thus, the flops of this part are on the order of $2 n^d r + o(n^d r)$. 
    
Then
\begin{equation*}
    \begin{aligned}
    \scr{P}_{\hat{T}} (\scr{M}_1(\bd{\mcal{Z}}))& = \underbrace{\left[\begin{array}{cc}
    \mbf{\hat{U}} & \mbf{Q}_1 
\end{array}\right]}_{n\times (r + r)} \underbrace{\left[\begin{array}{cc}
    \mbf{\hat{U}}^T \scr{M}_{1}(\bd{\mcal{Z}}) \mbf{\hat{V}} &  \mbf{R}_2^T \\
    \mbf{R}_1 & 0
\end{array}\right]}_{(r + r)\times(r + r)} \underbrace{\left[\begin{array}{c}
     \mbf{\hat{V}}^T\\
     \mbf{Q}_2^T
\end{array}\right]}_{(r + r)\times n^{d-1}}\\
&:=\left[\begin{array}{cc}
    \mbf{\hat{U}} & \mbf{Q}_1 
\end{array}\right] \mbf{M} \left[\begin{array}{c}
     \mbf{\hat{V}}^T\\
     \mbf{Q}_2^T
\end{array}\right].
    \end{aligned}
    \end{equation*}
    
Since $\left[\begin{array}{cc}
    \mbf{\hat{U}} & \mbf{Q}_1 
\end{array}\right]$ and $ \left[\begin{array}{cc}
    \mbf{\hat{V}} & \mbf{Q}_2 
\end{array}\right]$ are both orthogonal matrices, the truncated SVD of $\scr{P}_{\hat{T}} (\scr{M}_1(\bd{\mcal{Z}}))$ can be obtained from truncated SVD of $\mbf{M}$. Denote the $r$-truncation of $\mbf{M}$ as $\scr{H}_{r}(\mbf{M}) = \mbf{U_M} \mbf{\Sigma_M} \mbf{V_M}^T$, requiring $O(r^3)$ flops. Therefore, 
for the computation of $\scr{H}^{\text{ST}}_{\mbf{r}} $ in Algorithm \ref{alg: ST-HOSVD},
\begin{equation*}
\begin{aligned}
\mbf{U}_1 &= \left[\begin{array}{cc}
    \mbf{\hat{U}} & \mbf{Q}_1 
\end{array}\right] \mbf{U_M}\\
\bd{\mcal{B}} &= \scr{M}_1^{-1}\left(\mbf{\Sigma_M} \mbf{V_M}^T \left[\begin{array}{c}
     \mbf{\hat{V}}^T\\
     \mbf{Q}_2^T
\end{array}\right]\right).
\end{aligned}
\end{equation*}
    
Here, the calculation of  $\mbf{U}_1$ involves the multiplication of a $n\times 2 r$ matrix and a $2 r\times r$ matrix, requiring $2 n r^2$ flops. While $\bd{\mcal{B}}$ can be obtained by first multiplying $\mbf{\Sigma_M}$ by $\mbf{V_M}^T$, requiring  $2 r^3$ flops, and then multiplying $(\mbf{\Sigma_M} \mbf{V_M}^T)$ by $\left[\begin{array}{c}
     \mbf{\hat{V}}^T\\
     \mbf{Q}_2^T
\end{array}\right]$, with additional $2 n^{d-1} r^2$ flops. Thus, the first mode calculation in Algorithm \ref{alg: ST-HOSVD} requires $2 (n^{d-1}r^2 + n r^2 + r^3) = o(n^d r)$ flops in total. For the other modes in Algorithm \ref{alg: ST-HOSVD}, we need $O(\sum\limits_{k = 2}^d n^{d+1-k} r^k)$ flops to obtain the factor matrices $\mbf{U}_i, i = 1,\dots, d$ and resize the tensor $\bd{\mcal{B}}$. And, it requires $n^{d} r + o(n^d r)$ flops to obtain the composition $\scr{H}_{\mbf{r}}^{\text{ST}}\circ\scr{P}_{\hat{\bb{T}}}(\bd{\mcal{Z}}) = \bd{\mcal{B}}\times_1 \mbf{U}_1\times_2\cdots\times_d \mbf{U}_d$.

\vskip 2mm
In summary, the overall computational complexity of $\scr{H}_{\mbf{r}}^{\text{ST}}\circ\scr{P}_{\hat{\bb{T}}}$ is on the order of $3 n^d r + o(n^d r)$ and the highest order comes from the matrix multiplication which is easy to speed up by parallel computing.  On the contrary, directly computing $\scr{H}_{\mbf{r}}^{\text{ST}}(\bd{\mcal{Z}})$ involves the computation of the truncated-$r$ SVD of $n\times n^{d-1}$ matrix and the final tensor product $\bd{\mcal{B}}\times_1 \mbf{U}_1\times_2\cdots\times_d \mbf{U}_d$, it typically costs $O(n^d r)$ flops, but with a large hidden constant in front of $n^d r$.

\item \textbf{Computation of} $\alpha_k$.

The computational cost of $\alpha_k$ depends on the computation of $\scr{P}_{\hat{\bb{T}}}(\bd{\mcal{G}})$, as defined in \eqref{equ: tangent space projection}
\begin{equation*}
 \begin{aligned}
\scr{P}_{\hat{T}}(\scr{M}_1(\bd{\mcal{G}})) &= \mbf{\hat{U}} \mbf{\hat{U}}^T \scr{M}_1(\bd{\mcal{G}}) + \scr{M}_1(\bd{\mcal{G}}) \mbf{\hat{V}} \mbf{\hat{V}}^T - \mbf{\hat{U}} \mbf{\hat{U}}^T \scr{M}_1(\bd{\mcal{G}}) \mbf{\hat{V}} \mbf{\hat{V}}^T\\
& = \mbf{\hat{U}} \mbf{\hat{U}}^T \scr{M}_1(\bd{\mcal{G}}) + \left(\scr{M}_1(\bd{\mcal{G}}) - \mbf{\hat{U}} \mbf{\hat{U}}^T \scr{M}_1(\bd{\mcal{G}})\right)\mbf{\hat{V}} \mbf{\hat{V}}^T.
\end{aligned}
\end{equation*}
    
To compute $\scr{P}_{\hat{\bb{T}}}(\bd{\mcal{G}})$ with the lowest cost, for $\mbf{\hat{U}} \mbf{\hat{U}}^T \scr{M}_1(\bd{\mcal{G}})$, we can first compute $\mbf{\hat{U}}^T \scr{M}_1(\bd{\mcal{G}})$, which involves a matrix multiplication between an $r\times n$ matrix and an $n\times n^{d-1}$ matrix thus it requires $n^d r$ flops, then we compute $\mbf{\hat{U}} (\mbf{\hat{U}}^T \scr{M}_1(\bd{\mcal{G}}))$, which also requires $n^d r$ flops. While the subtraction $\left(\scr{M}_1(\bd{\mcal{G}}) - \mbf{\hat{U}} \mbf{\hat{U}}^T \scr{M}_1(\bd{\mcal{G}})\right)$ can be computed with a matrix subtraction operation. After that, we compute $\left(\scr{M}_1(\bd{\mcal{G}}) - \mbf{\hat{U}} \mbf{\hat{U}}^T \scr{M}_1(\bd{\mcal{G}})\right) \mbf{\hat{V}}$, it involves a matrix multiplication between an $n\times n^{d-1}$ matrix and an $n^{d-1}\times r$ matrix, requiring $n^d r$ flops. Finally, the computation of $\left(\scr{M}_1(\bd{\mcal{G}}) - \mbf{\hat{U}}\mbf{\hat{U}}^T \scr{M}_1(\bd{\mcal{G}})\right) \mbf{\hat{V}} \mbf{\hat{V}}^T$ requires additional $n^d r$ flops. Hence, the computation of $\scr{P}_{\hat{\bb{T}}}(\bd{\mcal{G}})$ requires $4n^d r + o(n^d r)$ flops. Based on that, both $\bnm\scr{P}_{\hat{\bb{T}}}(\bd{\mcal{G}})\bnm_F^2$ and $\bnm\scr{A}\scr{P}_{\hat{\bb{T}}}(\bd{\mcal{G}})\bnm_F^2$ require $O(n^d)$ flops. Therefore, the computation of $\alpha_k$ requires $4n^d r + o(n^d r)$ flops in total.
\end{itemize}
\vskip 2mm
In NIHT and SeMPIHT, the computation of the stepsize $\alpha$ depends on $\scr{F}(\bd{\mcal{G}})$ in \eqref{equ: IHT stepsize}, which involves the calculation of $\bd{\mcal{G}}\times_{i\in [d]}(\mbf{U}_i \mbf{U}_i^T)$. To implement this efficiently, one can first compute $\bd{\mcal{G}}\times_{i\in [d]}\mbf{U}_i^T$, and then compute $(\bd{\mcal{G}}\times_{i\in [d]}\mbf{U}_i^T)\times_{i\in [d]}\mbf{U}_i$. Thus, the overall computational cost for computing $\alpha_k$ is $2n^d r + o(n^d r)$. The computational complexity of RGD is analyzed in \cite{cai2020provable}. 

\subsection{Proofs of Main Results}
\subsubsection{Useful Lemmas}
Before proceeding to the main results, we first introduce some useful lemmas.

\begin{lemma}\label{lemma: tsp bound}
   \cite{wei2016guarantees} Let $\mbf{\hat{X}}=\mbf{\hat{U}} \mbf{\hat{\Sigma}} \mbf{\hat{V}}^T$ and $\mbf{X}= \mbf{U} \mbf{\Sigma} \mbf{V}^T$ be two rank $r$ matrices, then the following four inequalities hold
    \begin{equation*}
    \bnm\mbf{\hat{U}} \mbf{\hat{U}}^T-\mbf{U} \mbf{U}^T\bnm_2 \leq \frac{\bnm\mbf{\hat{X}}-\mbf{X}\bnm_2}{\sigma_{r}(\mbf{X})} \quad \text { and } \quad\bnm\mbf{\hat{V}} \mbf{\hat{V}}^T-\mbf{V} \mbf{V}^T\bnm_2 \leq \frac{\bnm\mbf{\hat{X}}-\mbf{X}\bnm_2}{\sigma_{r}(\mbf{X})};
    \end{equation*}
    \begin{equation*}
    \bnm\mbf{\hat{U}} \mbf{\hat{U}}^T-\mbf{U} \mbf{U}^T\bnm_F \leq \frac{\sqrt{2}\bnm\mbf{\hat{X}}-\mbf{X}\bnm_F}{\sigma_{r}(\mbf{X})} \quad \text { and } \quad\bnm\mbf{\hat{V}} \mbf{\hat{V}}^*-\mbf{V} \mbf{V}^*\bnm_F \leq \frac{\sqrt{2}\bnm\mbf{\hat{X}}-\mbf{X}\bnm_F}{\sigma_{r}(\mbf{X})}.
    \end{equation*}
    
where $\sigma_r(\mbf{X})$ denotes the $r^{\mathrm{th}}$ singular value of $\mbf{X}$.
\end{lemma}
\begin{lemma}\label{lemma: tsp error}
    Let $\mbf{\hat{X}}=\mbf{\hat{U}} \mbf{\hat{\Sigma}} \mbf{\hat{V}}^T$ be a rank $r$ matrix with $\hat{T}$ be its tangent space, and let $\mbf{X}$ be another matrix with rank $r$, then
   \begin{equation*}
    \bnm\left(\scr{I}-\scr{P}_{\hat{T}}\right) \mbf{X}\bnm_F  \leq \frac{1}{\sigma_{r }(\mbf{X})}\bnm\mbf{\hat{X}}-\mbf{X}\bnm_F^2.
    \end{equation*}
    
\end{lemma}
\begin{proof}
Denote the tangent space of $\mbf{X}$ as $T$, since $\scr{P}_{T}(\mbf{X}) = \mbf{X}$, then by the definition of tangent space projection formula in \eqref{equ: tangent space projection}, one has
    \begin{equation*}
    \begin{aligned}
     \bnm\left(\scr{I}-\scr{P}_{\hat{T}}\right) \mbf{X}\bnm_F & = \bnm\left(\scr{P}_{T}-\scr{P}_{\hat{T}}\right) \mbf{X}\bnm_F\\
    & = \bnm\mbf{U} \mbf{U}^T \mbf{X} -(\mbf{\hat{U}}\mbf{\hat{U}}^T \mbf{X} + \mbf{X} \mbf{\hat{V}} \mbf{\hat{V}}^T - \mbf{\hat{U}}\mbf{\hat{U}}^T \mbf{X} \mbf{\hat{V}} \mbf{\hat{V}}^T)\bnm_F\\
    & = \bnm(\mbf{U} \mbf{U}^T - \mbf{\hat{U}}\mbf{\hat{U}}^T)\mbf{X} - (\mbf{U} \mbf{U}^T \mbf{X}\mbf{\hat{V}} \mbf{\hat{V}}^T - \mbf{\hat{U}}\mbf{\hat{U}}^T \mbf{X} \mbf{\hat{V}} \mbf{\hat{V}}^T) \bnm_F\\
    & = \bnm(\mbf{U} \mbf{U}^T - \mbf{\hat{U}}\mbf{\hat{U}}^T)\mbf{X}(\mbf{I} - \mbf{\hat{V}}\mbf{\hat{V}}^T)\bnm_F\\
    & = \bnm(\mbf{U} \mbf{U}^T - \mbf{\hat{U}}\mbf{\hat{U}}^T)(\mbf{X} - \mbf{\hat{X}})(\mbf{I} - \mbf{\hat{V}}\mbf{\hat{V}}^T)\bnm_F \\
    & \leq \bnm\mbf{U} \mbf{U}^T - \mbf{\hat{U}}\mbf{\hat{U}}^T\bnm_2\bnm\mbf{X} - \mbf{\hat{X}}\bnm_F\bnm\mbf{I} - \mbf{\hat{V}}\mbf{\hat{V}}^T\bnm_2 \\
    &\leq \frac{1}{\sigma_{r }(\mbf{X})}\bnm\mbf{\hat{X}}-\mbf{X}\bnm_2\bnm\mbf{\hat{X}}-\mbf{X}\bnm_F \\
    & \leq \frac{1}{\sigma_{r}(\mbf{X})}\bnm\mbf{\hat{X}}-\mbf{X}\bnm_F^2,
    \end{aligned}
    \end{equation*}
    where the second equality follows from the fact $\mbf{U} \mbf{U}^T \mbf{X} = \mbf{X}$ and \eqref{equ: tangent space projection}. The last equality follows from the fact $\mbf{\hat{X}}(\mbf{I}- \mbf{\hat{V}}\mbf{\hat{V}}^T) = 0$. The second inequality follows from Lemma \ref{lemma: tsp bound} and the fact that  $\bnm\mbf{I} - \mbf{\hat{V}}\mbf{\hat{V}}^T\bnm_2\leq 1$. The last inequality follows from the fact that $\bnm\mbf{Z}\bnm_2\leq \bnm\mbf{Z}\bnm_F$ for given matrix $\mbf{Z}$.
\end{proof}

\begin{corollary}\label{cor: tsp error}
The distance before and after mode-1 tangent space projection of a tensor can be estimated as follows:
\begin{equation*}
    \begin{aligned}
\bnm(\scr{I} -  \scr{P}_{\hat{\bb{T}}^k})(\bd{\mcal{Z}})\bnm_F  &= \bnm(\scr{I} -  \scr{P}_{\hat{{T}}^k})(\scr{M}_1(\bd{\mcal{Z}}))\bnm_F \\
& \leq {\frac{1}{\sigma_{r_1 }(\scr{M}_1(\bd{\mcal{Z}}))} \bnm\scr{M}_1(\bd{\hat{\mcal{X}}}^k) - \scr{M}_1(\bd{\mcal{Z}})\bnm_F^2}\\
& = \frac{1}{\sigma_{r_1 }(\scr{M}_1(\bd{\mcal{Z}}))} \bnm\bd{\hat{\mcal{X}}}^k - \bd{\mcal{Z}}\bnm_F^2.
\end{aligned}
\end{equation*}

\end{corollary}

\vskip 3mm

\begin{lemma}\label{lemma: substitution error}
For the iterates $\bd{\mcal{X}}^k$, $\bd{\hat{\mcal{X}}}^k $ and $\bd{\mcal{W}}^k$ in Algorithm \ref{alg: SMQRGD}. The following inequalities hold
\begin{equation*}
    \begin{aligned}
\bnm\bd{\mcal{X}}^k - \bd{\mcal{T}}\bnm_F&\leq (\sqrt{d}+1) \bnm\bd{\mcal{W}}^k- \bd{\mcal{T}}\bnm_F; \\
\bnm\bd{\hat{\mcal{X}}}^k - \bd{\mcal{T}}\bnm_F&\leq 2 \bnm\bd{\mcal{W}}^k- \bd{\mcal{T}}\bnm_F;\\
\bnm\bd{\hat{\mcal{X}}}^k - \bd{\mcal{X}}^k\bnm_F&\leq (\sqrt{d}+1)\bnm\bd{\mcal{W}}^k- \bd{\mcal{T}}\bnm_F.
\end{aligned}
\end{equation*}

\end{lemma}
\begin{proof}
\begin{itemize}
    \item For  $\bnm\bd{\mcal{X}}^k - \bd{\mcal{T}}\bnm_F$, it follows that
    \begin{equation*}
    \begin{aligned}
    \bnm\bd{\mcal{X}}^k - \bd{\mcal{T}}\bnm_F
     & = \bnm\scr{H}^{ST}_{\mathbf{r}}(\bd{\mcal{W}}^k)  - \bd{\mcal{T}}\bnm_F \\
     & \leq \bnm\scr{H}^{ST}_{\mathbf{r}}(\bd{\mcal{W}}^k) - \bd{\mcal{W}}^k\bnm_F + \bnm\bd{\mcal{W}}^k - \bd{\mcal{T}}\bnm_F \\
     & \leq (\sqrt{d}+1) \bnm\bd{\mcal{W}}^k- \bd{\mcal{T}}\bnm_F,
    \end{aligned}
    \end{equation*}
    where the first inequality follows from $\bd{\mcal{X}}^k = \scr{H}^{ST}_{\mathbf{r}}(\bd{\mcal{W}}^k)$ and triangular inequality. The second inequality follows from Proposition \ref{prop: quasipro}.
    \item For $\bnm\bd{\hat{\mcal{X}}}^k - \bd{\mcal{T}}\bnm_F$, it follows that 
    \begin{equation*}
    \begin{aligned}
    \bnm\bd{\hat{\mcal{X}}}^k - \bd{\mcal{T}}\bnm_F
    & = \bnm\scr{H}_{r_1}(\bd{\mcal{W}}^k) - \bd{\mcal{T}}\bnm_F \\
    & \leq \bnm\scr{H}_{r_1}(\bd{\mcal{W}}^k) - \bd{\mcal{W}}^k\bnm_F + \bnm\bd{\mcal{W}}^k - \bd{\mcal{T}}\bnm_F \\
    & \leq 2 \bnm\bd{\mcal{W}}^k- \bd{\mcal{T}}\bnm_F,
    \end{aligned}
    \end{equation*}
     where the first inequality uses the triangular inequality and the second inequality uses the fact that $\bnm\scr{H}_{r_1}(\bd{\mcal{W}}^k) - \bd{\mcal{W}}^k\bnm_F\leq \bnm\bd{\mcal{W}}^k - \bd{\mcal{T}}\bnm_F$, since $\scr{H}_{r_1}(\bd{\mcal{W}}^k$ is the projection of $\bd{\mcal{W}}^k$ on the set consisting of tensor with rank $r_1$ in the first mode. 
    \item For $\bnm\bd{\hat{\mcal{X}}}^k - \bd{\mcal{X}}^k\bnm_F$, it follows that
     \begin{equation*}
    \begin{aligned}
    \bnm\bd{\hat{\mcal{X}}}^k - \bd{\mcal{X}}^k\bnm_F &= \bnm\scr{H}_{r_1}(\bd{\mcal{W}}^k) - \scr{H}^{ST}_{\mathbf{r}}(\bd{\mcal{W}}^k)\bnm_F\\
    &\leq \bnm\scr{H}_{r_1}(\bd{\mcal{W}}^k) - \bd{\mcal{W}}^k\bnm_F + \bnm\scr{H}^{ST}_{\mathbf{r}}(\bd{\mcal{W}}^k) - \bd{\mcal{W}}^k\bnm_F\\
    &\leq \bnm\bd{\mcal{W}}^k - \bd{\mcal{T}}\bnm_F + \sqrt{d} \bnm\bd{\mcal{W}}^k - \bd{\mcal{T}}\bnm_F \\
    & = (\sqrt{d}+1) \bnm\bd{\mcal{W}}^k - \bd{\mcal{T}}\bnm_F,
    \end{aligned}
    \end{equation*}
    the first inequality uses the triangular inequality and the second follows from the Proposition \ref{prop: quasipro}.
\end{itemize}
\end{proof}

\begin{lemma}\label{lemma: RIP bound of inner product}
\cite{wei2016guarantees} Let $\bd{\mcal{Z}}_1, \bd{\mcal{Z}}_2 \in \bb{R}^{n_1 \times n_2 \times \cdots \times n_d}$ be two mode-$1$ low-rank tensors. Suppose $\left\langle \bd{\mcal{Z}}_1, \bd{\mcal{Z}}_2\right\rangle = 0$ and $\rank_1\left(\bd{\mcal{Z}}_1\right)+\rank_1\left(\bd{\mcal{Z}}_2\right) \leq  n_1$, then
\begin{equation*}
    \left|\left\langle\scr{A}\left(\bd{\mcal{Z}}_1\right), \scr{A}\left(\bd{\mcal{Z}}_2\right)\right\rangle\right| \leq R_{\rank_1\left(\bd{\mcal{Z}}_1\right)+\rank\left(\bd{\mcal{Z}}_2\right)}\bnm\bd{\mcal{Z}}_1\bnm_F\bnm\bd{\mcal{Z}}_2\bnm_F,
\end{equation*}

where $R_{\rank_1\left(\bd{\mcal{Z}}_1\right) + \rank\left(\bd{\mcal{Z}}_2\right)}$ is the 1-RIC in Definition \ref{def: RIP property}.
\end{lemma}

\begin{lemma}\label{Iniest}
For $\bd{\mcal{X}}^0 = \scr{H}^{\text{ST}}_{\mathbf{r}}(\scr{A}^*\scr{A}(\bd{\mcal{T}}))$ and $\bd{\hat{\mcal{X}}}^0 = \scr{H}_{r_1}(\scr{A}^* \scr{A}(\bd{\mcal{T}}))$, then the following inequality holds:
    \begin{equation}\label{equ: initial substitute point error}
    \bnm\bd{\hat{\mcal{X}}}^0 - \bd{\mcal{T}}\bnm_F < 2 R_{2 r_1}\bnm\bd{\mcal{T}}\bnm_F;
    \end{equation}
    \begin{equation}\label{equ: initial error}
    \bnm\bd{\mcal{X}}^0 - \bd{\mcal{T}}\bnm_F < 2 (\sqrt{d-1}+1) R_{2 r_1}\bnm\bd{\mcal{T}}\bnm_F.
    \end{equation}
\end{lemma}

\begin{proof}
As $\bd{\mcal{X}}^0 = \scr{H}^{\text{ST}}_{\mathbf{r}}(\scr{A}^*\scr{A}(\bd{\mcal{T}}))$ and $\bd{\hat{\mcal{X}}}^0 = \scr{H}_{r_1}(\scr{A}^* \scr{A}(\bd{\mcal{T}}))$ in Algorithm \ref{alg: SMQRGD}, we have  
\begin{equation}\label{lm56eq1}
    \bnm\bd{\mcal{X}}^0 - \bd{\mcal{T}} \bnm_F\leq\bnm\bd{\mcal{X}}^0 -\bd{\hat{\mcal{X}}}^0 \bnm_F + \bnm\bd{\hat{\mcal{X}}}^0 - \bd{\mcal{T}}\bnm_F \leq (\sqrt{d-1}+1)\bnm\bd{\hat{\mcal{X}}}^0 - \bd{\mcal{T}}\bnm_F,
\end{equation}
where the second inequality uses $\bnm\bd{\mcal{X}}^0 -\bd{\hat{\mcal{X}}}^0\bnm_F\leq \sqrt{d-1} \bnm\bd{\hat{\mcal{X}}}^0 - \bd{\mcal{T}}\bnm_F$ (\cite{vannieuwenhoven2012new}, Theorem 6.4). Let $\mbf{Q} \in \bb{R}^{n_1\times 2 r_1}$ be the orthogonal matrix whose column spans the column space of $\scr{M}_1(\bd{\mcal{T}})$ and $\scr{M}_1(\bd{\hat{\mcal{X}}}^0)$ with $\mbf{Q}^{\perp}$ be its orthogonal completion matrix. Denote $\scr{P}_{\mbf{Q}}(\cdot)$ and $\scr{P}_{\mbf{Q}^{\perp}}(\cdot) $ as the projection operators on column space of  $\mbf{Q}$ and $\mbf{Q}^{\perp}$, then we have
$\bd{\hat{\mcal{X}}}^0 = \scr{P}_{\mbf{Q}}(\bd{\hat{\mcal{X}}}^0)$, $\scr{P}_{\mbf{Q}}(\bd{\mcal{T}}) = \bd{\mcal{T}}$, and the following equations hold
\begin{equation}\label{orth}
\bnm\bd{\hat{\mcal{X}}}^0 - \scr{A}^*\scr{A}(\bd{\mcal{T}})\bnm_F^2 = \bnm \bd{\hat{\mcal{X}}}^0 - \scr{P}_\mbf{Q}\scr{A}^*\scr{A}(\bd{\mcal{T}})\bnm_F^2 + \bnm\scr{P}_{\mbf{Q}^{\perp}}\scr{A}^*\scr{A}(\bd{\mcal{T}})\bnm_F^2;
\end{equation}
and 
\begin{equation*}
\bnm \bd{\mcal{T}} - \scr{A}^*\scr{A}(\bd{\mcal{T}})\bnm_F^2 = \bnm \bd{\mcal{T}} - \scr{P}_\mbf{Q}\scr{A}^*\scr{A}(\bd{\mcal{T}})\bnm_F^2 + \bnm\scr{P}_{\mbf{Q}^{\perp}}\scr{A}^*\scr{A}(\bd{\mcal{T}})\bnm_F^2.
\end{equation*}
Since $\bd{\hat{\mcal{X}}}^0$ is the projection of $\scr{A}^*\scr{A}(\bd{\mcal{T}})$, we have  $\bnm \bd{\hat{\mcal{X}}}^0 - \scr{A}^*\scr{A}(\bd{\mcal{T}})\bnm_F\leq \bnm \bd{\mcal{T}} - \scr{A}^*\scr{A}(\bd{\mcal{T}})\bnm_F$, then from \eqref{orth} it follows that 
\begin{equation*}
    \bnm \bd{\hat{\mcal{X}}}^0 - \scr{P}_\mbf{Q}\scr{A}^*\scr{A}(\bd{\mcal{T}})\bnm_F\leq \bnm \bd{\mcal{T}} - \scr{P}_\mbf{Q}\scr{A}^*\scr{A}(\bd{\mcal{T}})\bnm_F.
\end{equation*}
Therefore,
\begin{equation*}
    \begin{aligned}
\bnm\bd{\hat{\mcal{X}}}^0 - \bd{\mcal{T}}\bnm_F&\leq \bnm \bd{\hat{\mcal{X}}}^0 - \scr{P}_\mbf{Q}\scr{A}^*\scr{A}(\bd{\mcal{T}})\bnm_F + \bnm \bd{\mcal{T}} - \scr{P}_\mbf{Q}\scr{A}^*\scr{A}(\bd{\mcal{T}})\bnm_F\\
&\leq 2 \bnm \bd{\mcal{T}} - \scr{P}_\mbf{Q}\scr{A}^*\scr{A}(\bd{\mcal{T}})\bnm_F\\
& = 2\bnm(\scr{P}_\mbf{Q} - \scr{P}_\mbf{Q}\scr{A}^*\scr{A}\scr{P}_\mbf{Q}) (\bd{\mcal{T}}) \bnm_F,\\
\end{aligned}
\end{equation*}
where the equality follows from $\bd{\mcal{T}}=\scr{P}_\mbf{Q}(\bd{\mcal{T}})$.\\
To estimate $\scr{P}_\mbf{Q} - \scr{P}_\mbf{Q}\scr{A}^*\scr{A}\scr{P}_\mbf{Q}$, since $\rank_1(\scr{P}_\mbf{Q}(\bd{\mcal{Z}}))\leq 2 r_1$ for any tensor $\bd{\mcal{Z}}$, we have
\begin{equation*}
\begin{aligned}
\bnm \scr{P}_\mbf{Q} - \scr{P}_\mbf{Q}\scr{A}^*\scr{A}\scr{P}_\mbf{Q}\bnm & = \sup_{\|\bd{\mcal{Z}}\|_F = 1}\left|\langle  (\scr{P}_\mbf{Q} - \scr{P}_\mbf{Q}\scr{A}^*\scr{A}\scr{P}_\mbf{Q})(\bd{\mcal{Z}}), \bd{\mcal{Z}} \rangle \right| \\
& = \sup_{\|\bd{\mcal{Z}}\|_F = 1} \left|\bnm\scr{P}_\mbf{Q}(\bd{\mcal{Z}})\bnm_F^2 - \bnm\scr{A}\scr{P}_\mbf{Q} (\bd{\mcal{Z}}) \bnm_F^2 \right|\\
&\leq \sup_{\|\bd{\mcal{Z}}\|_F = 1} R_{2r_1} \bnm\scr{P}_\mbf{Q}(\bd{\mcal{Z}})\bnm_F^2\\
&\leq  R_{2r_1},
\end{aligned}
\end{equation*}
where the first inequality uses the $1$-RIP condition of operator $\scr{A}$.  Thus, $\bnm\bd{\hat{\mcal{X}}}^0 - \bd{\mcal{T}}\bnm_F$ can be bounded as the follows
\begin{equation}\label{lm56eq2}
    \bnm\bd{\hat{\mcal{X}}}^0 - \bd{\mcal{T}}\bnm_F\leq 2 R_{2 r_1}\bnm\bd{\mcal{T}}\bnm_F.
\end{equation}
Combine \eqref{lm56eq1} and \eqref{lm56eq2}, we have
\begin{equation*}
\bnm\bd{\mcal{X}}^0 - \bd{\mcal{T}}\bnm_F\leq 2 (\sqrt{d-1}+1) R_{2 r_1}\bnm\bd{\mcal{T}}\bnm_F.
\end{equation*}
\end{proof}

Before proving Theorem \ref{thm: alpha =1} and \ref{thm: SMQRGD}, we first give the one-step estimate of the SM-QRGD algorithm. \\
Recall the definition $\bd{\mcal{W}}^{k+1} = \scr{P}_{\hat{\bb{T}}^k} (\bd{\mcal{X}}^k - \alpha_k \scr{A}^*\scr{A}(\bd{\mcal{X}}^k - \bd{\mcal{T}}))$ and by Lemma \ref{lemma: substitution error}, one gets
\begin{equation*}
    \bnm\bd{\mcal{X}}^{k+1} - \bd{\mcal{T}}\bnm_F\leq (\sqrt{d}+1) \bnm\bd{\mcal{W}}^{k+1}- \bd{\mcal{T}}\bnm_F.
\end{equation*}
\noindent Therefore, it turn to estimate $\bnm\bd{\mcal{W}}^{k+1}- \bd{\mcal{T}}\bnm_F$. \\
According to the definition of $\bd{\mcal{W}}^{k+1}$, 
\begin{equation*}
\begin{aligned}
\bnm\bd{\mcal{W}}^{k+1} - \bd{\mcal{T}}\bnm_F &= \bnm\bd{\mcal{X}}^k - \alpha_k \scr{P}_{\hat{\bb{T}}^k} (\scr{A}^*\scr{A}(\bd{\mcal{X}}^k - \bd{\mcal{T}})) - \bd{\mcal{T}}\bnm_F\\
&\leq \underbrace{\bnm(\scr{P}_{\hat{\bb{T}}^k} - \alpha_k \scr{P}_{\hat{\bb{T}}^k}\scr{A}^* \scr{A}\scr{P}_{\hat{\bb{T}}^k})(\bd{\mcal{X}}^k-\bd{\mcal{T}})\bnm_F}_{I_{1}} + \underbrace{\bnm(\scr{I} -  \scr{P}_{\hat{\bb{T}}^k})(\bd{\mcal{T}})\bnm_F}_{I_{2}}\\
&\quad + \underbrace{\alpha_k  \bnm\scr{P}_{\hat{\bb{T}}^k}\scr{A}^* \scr{A}(\scr{I} - \scr{P}_{\hat{\bb{T}}^k})(\bd{\mcal{T}})\bnm_F}_{I_{3}}. \\
\end{aligned}
\end{equation*}
where the first equality uses Lemma \ref{lemma: lmPSTS}, i.e., $\scr{P}_{\hat{\bb{T}}^k}(\bd{\mcal{X}}^k) = \bd{\mcal{X}}^k$. \\
We now estimate the $I_1, I_2$ and $I_3$, separately.
\begin{itemize}
    \item For $I_{1}$, it follows that
    \begin{equation}\label{thmlmeq1}
    \begin{aligned}
    \bnm\scr{P}_{\hat{\bb{T}}^k} -   \scr{P}_{\hat{\bb{T}}^k}\scr{A}^* \scr{A}  \scr{P}_{\hat{\bb{T}}^k}\bnm &=  \sup_{\|\bd{\mcal{Z}}\|_F = 1} \left|\langle (\scr{P}_{\hat{\bb{T}}^k} -   \scr{P}_{\hat{\bb{T}}^k}\scr{A}^* \scr{A}  \scr{P}_{\hat{\bb{T}}^k})(\bd{\mcal{Z}}),\bd{\mcal{Z}}\rangle\right|\\
    &\leq \sup_{\|\bd{\mcal{Z}}\|_F = 1}\left|\bnm \scr{P}_{\hat{\bb{T}}^k}(\bd{\mcal{Z}})\bnm_F^2 - \bnm\scr{A}  \scr{P}_{\hat{\bb{T}}^k}(\bd{\mcal{Z}})\bnm_F^2\right|\\
    &\leq R_{2 r_1},
    \end{aligned}
    \end{equation}
    where the last inequality uses $1$-RIP condition of the operator $\scr{A}$.\\
    And 
    \begin{equation}\label{thmlmeq2}
    \begin{aligned}
    \bnm\scr{P}_{\hat{\bb{T}}^k}\scr{A}^* \scr{A} \scr{P}_{\hat{\bb{T}}^k}\bnm &= \sup_{\|\bd{\mcal{Z}}\|_F = 1}\left|\langle\scr{P}_{\hat{\bb{T}}^k}\scr{A}^* \scr{A} \scr{P}_{\hat{\bb{T}}^k}(\bd{\mcal{Z}}), \bd{\mcal{Z}}\rangle\right|\\
    &= \sup_{\|\bd{\mcal{Z}}\|_F = 1}\left|\langle \scr{A} \scr{P}_{\hat{\bb{T}}^k}(\bd{\mcal{Z}}), \scr{A} \scr{P}_{\hat{\bb{T}}^k}(\bd{\mcal{Z}})\rangle\right|\\
    &\leq \sup_{\|\bd{\mcal{Z}}\|_F = 1}(1+ R_{2 r_1})\bnm\scr{P}_{\hat{\bb{T}}^k}(\bd{\mcal{Z}})\bnm_F^2\leq 1 + R_{2 r_1}.
    \end{aligned}
    \end{equation}
    Therefore, $I_1$ can be bounded by
    \begin{equation*}
    \begin{aligned}
    I_{1}&\leq \bnm(\scr{P}_{\hat{\bb{T}}^k} - \scr{P}_{\hat{\bb{T}}^k}\scr{A}^* \scr{A}\scr{P}_{\hat{\bb{T}}^k})(\bd{\mcal{X}}^k - \bd{\mcal{T}})\bnm_F + |1-\alpha_k|\cdot\\
    &\quad \bnm \scr{P}_{\hat{\bb{T}}^k}\scr{A}^* \scr{A}\scr{P}_{\hat{\bb{T}}^k}(\bd{\mcal{X}}^k - \bd{\mcal{T}}) \bnm_F\\
    &\leq \left(R_{2 r_1} + |1-\alpha_k| (1+R_{2 r_1})\right)\bnm\bd{\mcal{X}}^k - \bd{\mcal{T}}\bnm_F\\
    &\leq (\sqrt{d}+1)\left(R_{2 r_1} + |1-\alpha_k| (1+R_{2 r_1})\right)\bnm\bd{\mcal{W}}^k - \bd{\mcal{T}}\bnm_F,
    \end{aligned}
    \end{equation*}
    where the last inequality follows from Lemma \ref{lemma: substitution error}.
    \item For $I_{2}$,
    by Lemma \ref{cor: tsp error} and Lemma \ref{lemma: substitution error}, one gets
     \begin{equation*}
    \begin{aligned}
    I_{2}& = \bnm(\scr{I} -  \scr{P}_{\hat{\bb{T}}^k})(\bd{\mcal{T}})\bnm_F \\
    &\leq \frac{1}{\sigma_{r_1 }(\scr{M}_1(\bd{\mcal{T}}))} \bnm\bd{\hat{\mcal{X}}}^k - \bd{\mcal{T}}\bnm_F^2\\
    &\leq \frac{4}{\sigma_{r_1 }(\scr{M}_1(\bd{\mcal{T}}))}\bnm\bd{\mcal{W}}^k - \bd{\mcal{T}}\bnm_F^2.
    \end{aligned}
    \end{equation*}

    \item For $I_{3}$,
    \begin{equation}\label{thmlmeq3}
    \begin{aligned}
    I_{3} &=  \alpha_k \sup_{\|\bd{\mcal{Z}}\|_F = 1}\left|\langle\scr{P}_{\hat{\bb{T}}^k}\scr{A}^* \scr{A} (\scr{I} - \scr{P}_{\hat{\bb{T}}^k})( \bd{\mcal{T}}), \bd{\mcal{Z}}\rangle\right|\\
    &= \alpha_k \sup_{\|\bd{\mcal{Z}}\|_F = 1} \left|\langle\scr{A} (\scr{I} - \scr{P}_{\hat{\bb{T}}^k})( \bd{\mcal{T}}),\scr{A} \scr{P}_{\hat{\bb{T}}^k}(\bd{\mcal{Z}})\rangle \right|\\
    &\leq \alpha_k \sup_{\|\bd{\mcal{Z}}\|_F = 1} R_{3 r_1}  \bnm(\scr{I} - \scr{P}_{\hat{\bb{T}}^k})(\bd{\mcal{T}})\bnm_F \bnm\scr{P}_{\hat{\bb{T}}^k}(\bd{\mcal{Z}})\bnm_F\\
    &\leq \alpha_k R_{3 r_1} \bnm \bd{\hat{\mcal{X}}}^k- \bd{\mcal{T}}\bnm_F\\
    & \leq 2 \alpha_k R_{3 r_1}\bnm \bd{\mcal{W}}^k- \bd{\mcal{T}}\bnm_F,
    \end{aligned}
    \end{equation}
    where the first inequality follows from Lemma \ref{lemma: RIP bound of inner product} with the fact that \\ 
    $\rank (\scr{M}_1((\scr{I} - \scr{P}_{\hat{\bb{T}}^k})( \bd{\mcal{T}})))\leq r_1$ and $\rank (\scr{M}_1(\scr{P}_{\hat{\bb{T}}^k}(\bd{\mcal{Z}})))\leq 2 r_1$. The second inequality follows the fact that $\bnm\scr{P}_{\hat{\bb{T}}^k}(\bd{\mcal{Z}})\bnm_F \leq \bnm\bd{\mcal{Z}}\bnm_F\leq 1$ and $\bnm(\scr{I} - \scr{P}_{\hat{\bb{T}}^k})(\bd{\mcal{T}})\bnm_F\leq \bnm \bd{\hat{\mcal{X}}}^k- \bd{\mcal{T}}\bnm_F$. The last inequality follows from Lemma \ref{lemma: substitution error}.
    
\end{itemize}
Finally, we have the following estimation
\begin{equation}
    \begin{aligned}\label{equ: sum up error estimation}
    \bnm\bd{\mcal{W}}^{k+1}- \bd{\mcal{T}}\bnm_F
    &\leq \Big((\sqrt{d}+1)\big(R_{2 r_1} + |1-\alpha_k| (1+R_{2 r_1})\big) + 2 \alpha_k R_{3r_1} \Big) \bnm\bd{\mcal{W}}^k - \bd{\mcal{T}}\bnm_F\\
    &\quad + \frac{4}{\sigma_{r_1}(\scr{M}_1(\bd{\mcal{T}}))} \bnm\bd{\mcal{W}}^k - \bd{\mcal{T}}\bnm_F^2.
    \end{aligned}
\end{equation}

\subsubsection{Proof of Theorem \ref{thm: alpha =1}}

% \vskip 2mm

% \textbf{Proof of Theorem \ref{thm: alpha =1}:}
\begin{proof}
Observing  $R_{2 r_1}\leq R_{3 r_1}$, $\alpha_k= 1$, and by \eqref{equ: sum up error estimation}, one has

\begin{equation}\label{equ: linear estimation}
    \bnm\bd{\mcal{W}}^{k+1} - \bd{\mcal{T}}\bnm_F\leq \left(\frac{4}{\sigma_{r_1 }(\scr{M}_1(\bd{\mcal{T}}))}\bnm\bd{\mcal{W}}^k - \bd{\mcal{T}}\bnm_F + (\sqrt{d}+3)R_{3 r_1}  \right)\bnm\bd{\mcal{W}}^k - \bd{\mcal{T}}\bnm_F.
\end{equation}

\noindent To ensure linear convergence, it suffices to let the coefficient of the right side of the above inequality strictly less than $1$, i.e.,
\begin{equation}\label{equ: constant stepsize, linear convergent 1}
 \left(\frac{4}{\sigma_{r_1 }(\scr{M}_1(\bd{\mcal{T}}))}\bnm\bd{\mcal{W}}^k - \bd{\mcal{T}}\bnm_F + (\sqrt{d}+3)R_{3 r_1}  \right) < 1.
\end{equation}
It is easy to verify that if \eqref{equ: constant stepsize, linear convergent 1} holds for $k = 1$, then it will hold for all $k\geq 2$ since \eqref{equ: linear estimation} yields $\bnm\bd{\mcal{W}}^{k+1} - \bd{\mcal{T}}\bnm_F\leq \bnm\bd{\mcal{W}}^k - \bd{\mcal{T}}\bnm_F$. \\
For $k = 1$,
\begin{equation}\label{them1ineq}
\begin{aligned}
\bnm\bd{\mcal{W}}^1 - \bd{\mcal{T}}\bnm_F &= \bnm\bd{\mcal{X}}^0 -  \scr{P}_{\hat{\bb{T}}^0}(\scr{A}^* \scr{A} (\bd{\mcal{X}}^0 - \bd{\mcal{T}})) - \bd{\mcal{T}} \bnm_F\\
&\leq {\bnm\bd{\mcal{X}}^0 - \bd{\mcal{T}}\bnm_F} +  \bnm\scr{P}_{\hat{\bb{T}}^0}\scr{A}^* \scr{A}\scr{P}_{\hat{\bb{T}}^0} (\bd{\mcal{X}}^0 - \bd{\mcal{T}})\bnm_F\\
&\quad +  \bnm\scr{P}_{\hat{\bb{T}}^0}\scr{A}^* \scr{A}(\scr{I} - \scr{P}_{\hat{\bb{T}}^0}) (\bd{\mcal{T}})\bnm_F\\
&\leq  \bnm\bd{\mcal{X}}^0 - \bd{\mcal{T}}\bnm_F +  (R_{2 r_1}+1) \bnm\bd{\mcal{X}}^0 - \bd{\mcal{T}}\bnm_F + R_{3 r_1} \bnm\bd{\hat{\mcal{X}}}^0 - \bd{\mcal{T}}\bnm_F \\
&< 2 R_{3 r_1}\left((\sqrt{d-1}+1) (R_{3 r_1}+2) + R_{3 r_1}\right)\bnm\bd{\mcal{T}}\bnm_F. 
\end{aligned}
\end{equation}
the first inequality uses the triangular inequality and Lemma \ref{lemma: lmPSTS}. The second inequality uses the estimates in \eqref{thmlmeq2} and \eqref{thmlmeq3}. The last inequality uses the Lemma \ref{Iniest} and fact that $R_{2 r_1} < R_{3 r_1}$.\\
Plug \eqref{them1ineq} into \eqref{equ: constant stepsize, linear convergent 1}, observing $\bnm \bd{\mcal{T}}\bnm_F\leq \sqrt{r_1} \sigma_1(\scr{M}_1(\bd{\mcal{T}}))$, 
it is sufficient to require
\begin{equation}\label{equ: final control constant}
     \gamma_1 :=  R_{3 r_1}\left(8\sqrt{r_1}\kappa_1\left((\sqrt{d-1}+1) (R_{3 r_1}+2) + R_{3 r_1}\right) + \sqrt{d} + 3\right) < 1,
\end{equation}
or equivalently 
\begin{equation}\label{equ: bd1}
\begin{aligned}
R_{3 r_1}
& <\frac{1}{8\sqrt{r_1}\kappa_1\left((\sqrt{d-1}+1) (R_{3 r_1}+2) + R_{3 r_1}\right) + \sqrt{d} + 3} \\
& =:\frac{1}{g_1(R_{3 r_1})},
\end{aligned}
\end{equation}
It is observed that $g_1(b)$ is continuous, increases in $(0, 1)$ and $g_1(0)$ is bounded, to ensure \eqref{equ: bd1} hold, it suffices to require $R_{3 r_1}< \min \left(b, \frac{1}{g_1(b)}\right)$, for any $b\in (0, 1)$, since
\begin{equation}\label{equ: interpolation estimation}
    R_{3 r_1}<\frac{1}{g_1(b)}< \frac{1}{g_1(R_{3 r_1})},
\end{equation}
where the second inequality follows from $R_{3 r_1}< b$ and $g_1$ is strictly increasing. \\
Take $b = \frac{1}{2}$, we get the sufficient condition in Theorem \ref{thm: alpha =1}
\begin{equation*}
R_{3 r_1} < \min \left(\frac{1}{2}, \frac{1}{ (20\sqrt{d-1} + 24)\sqrt{r_1}\kappa_1 + \sqrt{d} + 3} \right).
\end{equation*}
The linear convergence is obtained as the following
\begin{equation*}
\begin{aligned}
\bnm\bd{\mcal{W}}^{k+1} -\bd{\mcal{T}}\bnm_F
& \leq \gamma_1^k\bnm\bd{\mcal{W}}^1 - \bd{\mcal{T}}\bnm_F\\
& \leq  2 R_{3 r_1}\left((\sqrt{d-1}+1) (R_{3 r_1}+2) + R_{3 r_1}\right) \gamma_1^{k} \bnm\bd{\mcal{T}}\bnm_F \\
& \leq (\frac{5}{2}\sqrt{d-1}+3)\gamma_1^{k}\bnm\bd{\mcal{T}}\bnm_F.
\end{aligned}
\end{equation*}
and
\begin{equation*}
    \begin{aligned}
    \bnm\bd{\mcal{X}}^{k+1} - \bd{\mcal{T}}\bnm_F&\leq  (\sqrt{d}+1) \bnm\bd{\mcal{W}}^{k+1}- \bd{\mcal{T}}\bnm_F\\
    &\leq (\sqrt{d}+1)(\frac{5}{2} \sqrt{d-1} + 3) \gamma_1^{k}\bnm\bd{\mcal{T}}\bnm_F.
\end{aligned}
\end{equation*}
This completes the proof.
\end{proof}
\subsubsection{Proof of Theorem \ref{thm: SMQRGD}}
\begin{proof}
Likewise, for normalized step size \eqref{equ: step size choice}, one has
\begin{equation}\label{equ: stepsize estimation}
\frac{1}{1+R_{2 r_1}} \leq \alpha_k=\frac{\bnm\scr{P}_{\hat{\bb{T}}^k}(\bd{\mcal{G}}^k)\bnm_F^2}{\bnm\scr{A}(\scr{P}_{\hat{\bb{T}}^k}(\bd{\mcal{G}}^k))\bnm_F^2} \leq \frac{1}{1-R_{2 r_1}},
\end{equation}
which is equivalent to  
\begin{equation*}
    \left|\alpha_k-1\right| \leq \frac{R_{2 r_1}}{1 - R_{2 r_1}}.
\end{equation*}
Observing $R_{2 r_1}\leq R_{3 r_1}$, plugging the estimation of $\alpha_k$ into \eqref{equ: sum up error estimation} yields 
\begin{equation}\label{equ: linear convergent}
\bnm\bd{\mcal{W}}^{k+1} - \bd{\mcal{T}}\bnm_F\leq \left(\frac{4}{\sigma_{r_1 }(\scr{M}_1(\bd{\mcal{T}}))} \bnm\bd{\mcal{W}}^k - \bd{\mcal{T}}\bnm_F + (\sqrt{d}+2)\frac{2R_{3 r_1}}{1- R_{3 r_1}}\right)\bnm\bd{\mcal{W}}^k - \bd{\mcal{T}}\bnm_F.
\end{equation}
Analogously, it is sufficient to require the following inequality to obtain the linear convergence 
\begin{equation}\label{equ: linear convergent 1}
\left(\frac{4}{\sigma_{r_1 }(\scr{M}_1(\bd{\mcal{T}}))} \bnm\bd{\mcal{W}}^k - \bd{\mcal{T}}\bnm_F + (\sqrt{d}+2)\frac{2R_{3 r_1}}{1- R_{3 r_1}}\right)< 1.
\end{equation}
For $\bnm\bd{\mcal{W}}^1 - \bd{\mcal{T}}\bnm_F$,
\begin{equation*}
    \begin{aligned}
\bnm\bd{\mcal{W}}^1 - \bd{\mcal{T}}\bnm_F &= \bnm\bd{\mcal{X}}^0 - \alpha_0 \scr{P}_{\hat{\bb{T}}^0}(\scr{A}^* \scr{A} (\bd{\mcal{X}}^0 - \bd{\mcal{T}})) - \bd{\mcal{T}} \bnm_F\\
&\leq {\bnm\bd{\mcal{X}}^0 - \bd{\mcal{T}}\bnm_F} +  {\alpha_0 \bnm\scr{P}_{\hat{\bb{T}}^0}\scr{A}^* \scr{A}\scr{P}_{\hat{\bb{T}}^0} (\bd{\mcal{X}}^0 - \bd{\mcal{T}})\bnm_F} \\
&\quad + {\alpha_0 \bnm\scr{P}_{\hat{\bb{T}}^0}\scr{A}^* \scr{A}(\scr{I} - \scr{P}_{\hat{\bb{T}}^0}) (\bd{\mcal{T}})\bnm_F}\\
&\leq  {\bnm\bd{\mcal{X}}^0 - \bd{\mcal{T}}\bnm_F} + {\alpha_0 (R_{2 r_1}+1) \bnm\bd{\mcal{X}}^0 - \bd{\mcal{T}}\bnm_F} + \alpha_0 R_{3 r_1} \bnm\bd{\hat{\mcal{X}}}^0 - \bd{\mcal{T}}\bnm_F\\
&< \frac{2}{1-R_{3 r_1}} \bnm\bd{\mcal{X}}^0 - \bd{\mcal{T}}\bnm_F + \frac{R_{3 r_1}}{1 - R_{3 r_1}} \bnm\bd{\hat{\mcal{X}}}^0 - \bd{\mcal{T}}\bnm_F \\
&< \frac{2R_{3 r_1}}{1-R_{3 r_1}}\left( 2 (\sqrt{d-1}+1) + R_{3 r_1}\right)\bnm\bd{\mcal{T}}\bnm_F.
\end{aligned}
\end{equation*}
where the second to the last inequality uses the (\ref{equ: stepsize estimation}) and the last inequality uses Lemma \ref{Iniest}. \\
Plugging the above initialization error into \eqref{equ: linear convergent 1} for $k=1$ yields 
\begin{equation}\label{equ: final contral}
    \gamma_2 := \frac{2R_{3 r_1}}{1-R_{3 r_1}}\left(4 \sqrt{r_1}\kappa_1 \left(2 (\sqrt{d-1}+1) + R_{3 r_1}\right) + \sqrt{d}+2 \right)<1,
\end{equation}
or equivalently 
\begin{equation*}
    \begin{aligned}
    R_{3 r_1}&< \frac{1}{\frac{2}{1-R_{3 r_1}}\left(4 \sqrt{r_1}\kappa_1 \left(2 (\sqrt{d-1}+1) + R_{3 r_1}\right) + \sqrt{d}+2 \right)}\\
    & := \frac{1}{g_2(R_{3 r_1})}.
    \end{aligned}
\end{equation*}
Here $g_2(b)$ is continuous and increases in $(0,1)$, $g_2(0)$ is bounded. So if $R_{3 r_1}< \min \left(b, \frac{1}{g_2(b)}\right)$, for any $b\in (0, 1)$, then $ R_{3 r_1}<\frac{1}{g_2(b)}< \frac{1}{g_2(R_{3 r_1})}$ and \eqref{equ: final contral} holds. Taking $b = \frac{1}{2}$, we get the sufficient condition in Theorem \ref{thm: SMQRGD}:
\begin{equation*}
R_{3 r_1} < \min \left(\frac{1}{2}, \frac{1}{\left(32 \sqrt{d-1} + 40\right)\sqrt{r_1}\kappa_1  + 4 \sqrt{d}+ 8} \right).
\end{equation*}
Then, we have 
\begin{equation*}
\begin{aligned}
\bnm\bd{\mcal{W}}^{k+1}- \bd{\mcal{T}}\bnm_F
& \leq \gamma_2^{k} \bnm\bd{\mcal{W}}^1- \bd{\mcal{T}}\bnm_F \\
& \leq 2 R_{3 r_1}\frac{ 2 (\sqrt{d-1} + 1) + R_{3 r_1}}{1-R_{3 r_1}} \gamma_2^{k}\bnm\bd{\mcal{T}}\bnm_F \\ 
& \leq (8 \sqrt{d-1}+10)\gamma_2^{k}\bnm\bd{\mcal{T}}\bnm_F,
\end{aligned}
\end{equation*}
and 
\begin{equation*}
\begin{aligned}
\bnm\bd{\mcal{X}}^{k+1} - \bd{\mcal{T}}\bnm_F
& \leq  (\sqrt{d}+1)\bnm\bd{\mcal{W}}^{k+1}- \bd{\mcal{T}}\bnm_F \\
& \leq (\sqrt{d}+1)(8 \sqrt{d-1}+10)\gamma_2^{k} \bnm\bd{\mcal{T}}\bnm_F.
\end{aligned}
\end{equation*}
This completes the proof.
\end{proof}

\section*{Acknowledgments}
The authors would like to thank the anonymous referees very much for their careful reading and valuable comments, which significantly improved the quality of this manuscript. This work was supported by Shanghai Municipal Science and Technology Major Project (2021SHZDZX0102) and NSFC (No.12090024). We also thank the Student Innovation Center at Shanghai Jiao Tong University for providing the computing services.

\bibliographystyle{siamplain}
\bibliography{references}
\end{document}